\documentclass[12pt,leqno]{amsart}

\usepackage{amsmath,amssymb,amscd,amsthm}
\usepackage{latexsym}
\usepackage{graphicx}
\usepackage{fourier}

\graphicspath{{eps_files/}}

\newtheorem{theorem}{Theorem}
\newtheorem{lemma}{Lemma}
\newtheorem{proposition}{Proposition}

\newtheorem{definition}{Definition}
\newtheorem{corollary}{Corollary}

\newtheorem{claim}{Claim}

\newcommand{\f}[2]{\frac{#1}{#2}}
\newcommand{\dpr}[2]{\langle #1,#2 \rangle}

\newcommand{\ubb}{{\mathbf u}}
\newcommand{\uff}{{\mathbf f}}
\newcommand{\ubq}{{\mathbf q}}
\newcommand{\al}{\alpha}

\newcommand{\ga}{\gamma}

\newcommand{\de}{\delta}

\newcommand{\De}{\Delta}
\newcommand{\ve}{\varepsilon}

\newcommand{\la}{\lambda}

\newcommand{\si}{\sigma}

\newcommand{\vp}{\varphi}
\newcommand{\om}{\omega}

\newcommand{\rone}{\mathbf R^1}

\newcommand{\cf}{\mathcal F}

\newcommand{\cz}{\mathbb Z}

\newcommand{\ch}{\mathcal H}
\newcommand{\cl}{\mathcal L}

\newcommand{\p}{\partial}

\newcommand{\beq}{\begin{equation}}
\newcommand{\eeq}{\end{equation}}
\newcommand{\beqna}{\begin{eqnarray*}}
\newcommand{\eeqna}{\end{eqnarray*}}
\newcommand{\beqn}{\begin{equation*}}
\newcommand{\eeqn}{\end{equation*}}
\newcommand{\bp}{\begin{proof}}
\newcommand{\ep}{\end{proof}}
\newcommand{\bprop}{\begin{proposition}}
\newcommand{\eprop}{\end{proposition}}
\newcommand{\bt}{\begin{theorem}}
\newcommand{\et}{\end{theorem}}
\newcommand{\bex}{\begin{Example}}
\newcommand{\eex}{\end{Example}}
\newcommand{\bc}{\begin{corollary}}
\newcommand{\ec}{\end{corollary}}
\newcommand{\bcl}{\begin{claim}}
\newcommand{\ecl}{\end{claim}}
\newcommand{\bl}{\begin{lemma}}
\newcommand{\el}{\end{lemma}}

\usepackage{color}

\begin{document}

\title
[Ground states in spatially discrete non-linear Schr\"odinger lattices]
{Ground states in spatially discrete non-linear Schr\"odinger lattices}

\author{Atanas G. Stefanov}

\author{Ryan M. Ross}

\author{Panayotis G. Kevrekidis}

\address{Atanas Stefanov\\
Department of Mathematics \\
 University of Alabama - Birmingham\\ 
 University Hall 4049\\
 Birmingham, AL 35294}
\email{stefanov@uab.edu}

\address{Ryan Ross\\
Lederle Graduate Research Tower\\
Department of Mathematics and  Statistics\\
University of Massachusetts\\
Amherst, MA 01003}

\address{Panayotis G. Kevrekidis\\
Lederle Graduate Research Tower\\
Department of Mathematics and  Statistics\\
University of Massachusetts\\
Amherst, MA 01003}

\email{kevrekid@math.umass.edu}

\thanks{Stefanov's research is supported in part by
 NSF-DMS  1908626.  The present paper is based on work 
that was supported by the US National Science Foundation 
under DMS-1809074 and PHY-2110030 (P.G.K. and RMR).}
\date{\today}

\subjclass[2000]{???}


\begin{abstract}
  In the seminal work  \cite{Wein}, Weinstein considered the question
  of the ground states for discrete
  Schr\"odinger equations with power law nonlinearities, posed on ${\mathbb Z}^d$.   More specifically, he  constructed the so-called normalized waves, by minimizing the Hamiltonian functional, for fixed power $P$ 
  (i.e. $l^2$ mass). This type of  variational method allows one to
  claim, in a straightforward manner,  set stability for such waves.
  
%
 In this work,  we revisit  these  questions and build upon
 Weinstein's work in several directions. First, for the normalized
 waves,   we show that they are in fact spectrally stable as solutions
 of the corresponding discrete NLS evolution equation. Next, we
 construct the so-called homogeneous waves, by using a different
 constrained optimization problem. Importantly, this construction
 works for all values of the parameters, e.g. $l^2$ supercritical
 problems. We establish a rigorous criterion for stability, which
 decides the stability on the homogeneous waves, based  on the
 classical  Grillakis-Shatah-Strauss/Vakhitov-Kolokolov quantity
 $\p_\om \|\vp_\om\|_{l^2}^2$.     In addition, we provide some
 symmetry results for the solitons. Finally, we complement our results
 with numerical computations, which showcase the full agreement
 between the conclusion from the GSS/VK criterion   vis-{\' a}-vis with the linearized problem.  In particular, one observes that it is possible for the stability of the wave to change as the spectral parameter $\om$ varies, in contrast with the corresponding continuous NLS model. 
\end{abstract}

\maketitle

\section{Introduction and Motivation}

The discrete nonlinear Schr{\"o}dinger model~\cite{dnlsbook} has been
one of the workhorses within the realm of nonlinear dynamical lattice
models that has enabled the identification of numerous nonlinear
waveforms,
their stability analysis, instability manifestations and complex
dynamics
and thermodynamics. Arguably, central to its popularity can be thought
of as being the prototypical inclusion of the main ingredients for
such
phenomena, namely the interplay between nonlinearity and lattice
dispersion.
Another key of its features is its generic nature and multi-fold
physical
motivation stemming, originally, from the example of optical waveguide
arrays~\cite{dnc,moti},
but also extending nowadays to quite different fields, such as the
atomic
realm of Bose-Einstein condensates in optical lattices~\cite{ober}. 
As some among the numerous notable features that the theoretical analysis
and experimental investigations of the mode have enabled to
explore, we mention the manifestation of discrete
diffraction~\cite{yaron}
and diffraction managed solitons~\cite{yaron1,marka}, the illustration
of
lattice solitary waves~\cite{yaron2,yaron3}, the emergence of
discrete vortices~\cite{neshev,fleischer}, the Talbot
revivals~\cite{christo2},
the examination of $\mathcal{PT}$-symmetric lattices~\cite{ptrecent}.

At the same time, the DNLS has been a rich source of interesting
problems
in the applied mathematics literature and community. The early work
of~\cite{Wein} (that will be central herein) set the stage for some
of the important variational considerations that later extended
to both periodic and decaying solutions in similar
models~\cite{pankov},
but also for both focusing and defocusing
nonlinearities~\cite{herr1,herr2}
and continues to be an inspiration for further variational work on the
subject more recently~\cite{jia,hehe}. On the other hand, a
considerable
effort has been expended to obtain information about the
stability of different waveforms in one-~\cite{pkf1,saks},
two~\cite{pkf2}
and even three~\cite{pkf3} spatial dimensions. Additional
efforts have been made to address the spectral theory
and dispersive estimates for this model~\cite{stef1}, the
asymptotic stability of small solutions~\cite{stef2} , the distinction
between on- and off-site solutions~\cite{jenk1}, as well
as extensions more recently going beyond nearest-neighbor
interactions~\cite{jenk2,penati}.
While, admittedly, we cannot offer an exhaustive list, we believe that the above yields a
representative flavor of the range of contributions and interest
in the subject.

Here, we revisit the widely studied topic of the stability of
fundamental waveforms in the model, providing in the spirit
of~\cite{Wein} a novel constrained optimization perspective for
the so-called normalized waves pertaining to the ground state
of the system. On the one hand, the relevant formulation works
in the entire parametric regime of the system, while on the other
hand,
it offers an alternative yet rigorous criterion for the stability of
the solutions, which is tantamount to the famous Vakhitov-Kolokolov
criterion~\cite{vakhitov}, subsequently made rigorous for continuous
systems
in the work of Grillakis-Shatah-Strauss~\cite{grillakis} (see also the
recent exposition of~\cite{kapprom}). Our presentation of the relevant
results is structured as follows. In section 2, we provide the model
formulation and the results/theorems stemming from the
above reformulation. After the preliminaries of Section 3, we turn to
the construction of the normalized waves in Section 4. The resulting
existence and stability properties are presented in Section 5, while
Section 6 briefly summarizes our findings and presents our conclusions.

\section{The Model and Main Results}

We consider the focusing DNLS equation~\cite{dnlsbook} on $\cz^d,
d\geq 1$  of the form:
\begin{equation}
\label{10} 
i \p_t u_n+\De_{disc} u (n)+|u_n|^{2\si} u_n=0, u:\cz^d\to {\mathbb C}
\end{equation}
where 
$$
\De_{disc} u (n)=\sum_{j\in \cz^d: |j-n|=1} u_j - 2d u_n.
$$
The equation is well-known to conserve the Hamiltonian 
$$
H= \sum_{j\in \cz^d: |j-n|=1} |u_j-u_n|^2  - \f{1}{\si+1} \sum_{n\in\cz^d} |u_j|^{2\si+2} 
$$
Substituting the standard standing
wave ansatz $u_n=e^{i \om t} \vp_n$, we obtain the following difference equation, posed on $\cz^d$ 
\begin{equation}
\label{20} 
-\De_{disc} \vp_n+\om \vp_n - |\vp_n|^{2\si} \vp_n=0.
\end{equation}
Relevant quantities, which will be helpful in the sequel are the following 
\begin{eqnarray*}
P&=&\sum_{n\in\cz^d} |\vp_n|^2\\
V&=& \sum_{n\in\cz^d} |\vp_n|^{2\si+2} 
\end{eqnarray*}
It is also relevant to note in passing here that the squared $l^2$
norm is also a conserved quantity (mass) for the dynamics of
Eq.~(\ref{10}). 
Our aim in what follows will be, in line with many of the above
discussed works, to explore localized solutions of \eqref{20} and more
specifically
the fundamental discrete solitons with $\vp_n\geq 0$ that numerous
earlier
studies touched upon theoretically~\cite{Wein,jenk1} and
experimentally~\cite{yaron2}.
There is a number of (non-equivalent) ways in which one can introduce
such objects, but a quite natural  approach, that we adopt in this
work, is to consider such waves as appropriate (multiples of)
minimizers of appropriate variational problems. Even within this
framework, there is a number of ways one can do this, which affects
the stability of such waves
significantly. We  analyze  herein two  constrained variational
problems, which will provide us each
with a one-parameter family of solutions. 

We start by introducing the notion of normalized waves.
That is, for any fixed $l^2$ norm of the wave, we minimize the Hamiltonian $H$, subject to this constraint. This approach has been worked out, in some detail, in the work of Weinstein, \cite{Wein} - here we consider it again, as we need more specific properties regarding the stability of these waves.  More precisely, for a fixed $\la>0$, we solve the following variational problem: 
\begin{equation}
\label{30}
\left\{
\begin{array}{c}
H[{\mathbf u}]= \sum_{n\in\cz^d} \sum_{j\in \cz^d: |j-n|=1} |u_j-u_n|^2  - \f{1}{\si+1} \sum_{j\in\cz^d} |u_j|^{2\si+2} \to \min \\
\\
\|\ubb\|_{l^2}^2=\sum_{n\in \cz^d} |u_n|^2=\la.
\end{array}
\right.
\end{equation}
Let us emphasize right away that we do not expect to be able to solve \eqref{30}   for all values of $d, \si$. In fact, the variational problem \eqref{30} turns out to be ill-posed, i.e., 
$\inf_{\|\ubb\|_{l^2}^2=\la} H[u]=-\infty$, when $\si\geq \f{2}{d}$,
for all values of $\la$ small enough. This is, in fact, illustrated in
the original work of
Weinstein~\cite{Wein}. {More precisely,  he shows that for any $\si\geq \f{2}{d}$, there is $\la^*>0$, so that for all $0<\la<\la^*$, 
$$
\inf_{\|\ubb\|_{l^2}^2=\la} H[u]=0, 
$$
 and as it turns out, no constrained minimizers for \eqref{30} exist in this case}.

A different approach will be to construct the waves as (multiples of)  constrained minimizers of a   variational problem, with fixed potential energy.  That is, letting $\om>0$ be a fixed parameter (compare with \eqref{20}), we solve 
\begin{equation}
\label{40}
\left\{
\begin{array}{c}
 J[\ubb]:=  \sum_{n\in\cz^d}  \sum_{j\in \cz^d: |j-n|=1} |u_j-u_n|^2  +\om   \sum_{n\in\cz^d} |u_j|^2 \to \min \\
 \\ 
\|\ubb\|_{l^{2\si+2}}^{2\si+2}=\sum_{n\in \cz^d} |u_n|^{2\si+2}=1. 
\end{array}
\right.
\end{equation}
It is rather straightforward to  see that the problem \eqref{40}  is equivalent to the 
 problem for the minimization of the following un-constrained, but homogeneous functional 
$$
\tilde{J}[{\mathbf u}]=\inf\limits_{{\mathbf u}\neq 0}  
\f{J[{\mathbf u}]}{\left(\sum_{n\in \cz^d} |u_n|^{2\si+2}\right)^{\f{1}{\si+1}}}\to \min.
$$
It is worth noting that, as we shall establish later, the problem \eqref{40} is more flexible, in the sense that it allows a larger set of parameters, for which it produces non-trivial waves. Indeed, the requirement $\si<\f{2}{d}$ is no longer necessary and the waves   exist for all values of $0<\si<\infty, \om>0$. 

Next, we discuss the stability of these waves, as solutions to the DNLS \eqref{10}. More specifically, for solution $\vp$, consider a perturbation in the form $u_n(t)=e^{i \om t}(\vp_n+ e^{\la t}v_n)$. Plugging this in \eqref{10} and ignoring terms like $O(v^2)$, we obtain the linearized problem 
$$
i \la v_n+\De_{disc} v_n-\om v_n+\vp_n^{2\si} v_n +2\si \vp_n^{2\si} \Re v_n=0.
$$
Taking $v_n=(\Re v_n, \Im v_n)$, we obtain the following autonomous  problem for the perturbation $(\Re v_n, \Im v_n)$, 
\begin{equation}
\label{200} 
\left(\begin{array}{cc}
0 & -1 \\ 1 & 0 
\end{array}\right)\left(\begin{array}{cc}
\cl_+ & 0 \\ 0& \cl_- 
\end{array}\right) \left(\begin{array}{c}
\Re v_n \\ \Im v_n
\end{array}\right)=\la \left(\begin{array}{c}
\Re v_n \\ \Im v_n
\end{array}\right),
\end{equation}
where 
\begin{eqnarray*}
	\cl_+  &=& - \De_{disc}+ \om -(2\si+1) \vp_n^{2\si} \\
		\cl_-  &=& - \De_{disc}+ \om - \vp_n^{2\si}. 
\end{eqnarray*}
Accordingly, we say that the wave $e^{\i \om t}\vp_n$ is spectrally stable, if the eigenvalue problem \eqref{200} does not have a non-trivial solution $(\la, {\mathbf v}): \Re\la>0, \mathbf{v} \in l^2(\cz) \times l^2(\cz)$. The notion of orbital stability concerns the solution to the original non-linear problem, and it  is defined as follows - for every $\epsilon>0$, there exists $\de>0$, so that whenever the initial data ${\mathbf u}(0): \|{\mathbf u}(0)-\vp\|_{l^2}<\de$, then
$$
\sup_{t>0} \inf_{\theta\in\rone} \sum_n |u_n(t)-e^{i \theta} \vp_n|^2<\epsilon. 
$$
Note that the Cauchy problem for the full non-linear evolution,
\eqref{10} is not a concern in the discrete setting. More
specifically,   standard energy estimates show that the $l^2$
norm of the solutions of \eqref{10} is conserved, so that global solutions of
\eqref{10} exist whenever the initial data ${\mathbf u}(0)\in
l^2(\cz)$.

\subsection{Existence and stability for the normalized waves} 
We start with the general existence results for $d$ dimensional lattices and then we state more precise results for the case $d=1$. Our first result is about the normalized waves. 
\begin{theorem}(Construction of the normalized waves for $d$ dimensional lattices)
	\label{theo:5} 
	
	Let $d\geq 1$ and $0<\si<\f{2}{d}$. For each $\la>0$, the constrained variational problem \eqref{30} has a solution $\ubb$.  Any solution $\ubb$ is with non-negative entries $u_j$. Moreover, for every $j_0\in \{1, \ldots, d\}$ and fixed indices $k_1, \ldots, k_{j_0-1}, k_{j_0+1}, \ldots, k_d$, we have that there exist integers 
	$l_0, i_0$ depending on \\ $(k_1, \ldots, k_{j_0-1}, k_{j_0+1}, \ldots, k_d)$, so that 
	\begin{eqnarray}
	\label{ineq1}
	& & 	u_{k_1, \ldots, k_{j_0-1}, l_0, k_{j_0+1}, \ldots, k_d}=\ldots = u_{k_1, \ldots, k_{j_0-1}, l_0+i_0, k_{j_0+1}, \ldots, k_d}>u_{k_1, \ldots, k_{j_0-1}, l_0+i_0+1, k_{j_0+1}, \ldots, k_d}\ldots,  \\
	\label{ineq2}
	& & u_{k_1, \ldots, k_{j_0-1}, l_0, k_{j_0+1}, \ldots, k_d}>u_{k_1, \ldots, k_{j_0-1}, l_0-1, k_{j_0+1}, \ldots, k_d}\geq u_{k_1, \ldots, k_{j_0-1}, l_0-2, k_{j_0+1}, \ldots, k_d}\ldots. 
	\end{eqnarray}
	In the case $d=1$, and in addition to the \eqref{ineq1} and \eqref{ineq2}, we claim that the minimizer is symmetric in the following sense 
	$$
	u_0=u_1=\ldots=u_{i_0}>u_{i_0+1}=u_{-1} \geq u_{i_0+2}=u_{-2} \geq \ldots 
	$$
	Next, the minimizer $\ubb$ satisfies the Euler-Lagrange equation 
	\begin{equation}
	\label{EL:10} 
	-\De_{disc} u_n + c(\la) u_n - u_n^{2\si+1}=0, n\in \cz^d 
	\end{equation}
	for some Lagrange multiplier  $c(\la)>0$. Finally, the linearized discrete Schr\"odinger operator 
	$$
	\cl_+ f=-\De_{disc} f_n+c(\la) f_n - (2\si+1) u_n^{2\si} f_n
	$$
	enjoys the property $\cl_+|_{\{\ubb\}^\perp}\geq 0$. In particular, 
\end{theorem}
{\bf Remark:} 
\begin{itemize}
	\item As a straightforward  consequence of the property $\cl_+|_{\{\ubb\}^\perp}\geq 0$, we see that the Morse index $n(\cl_+)\leq 1$. On the other hand by Euler-Lagrange, \eqref{EL:10},  we obtain  
	$$
	\dpr{\cl_+ \ubb}{\ubb} = -2\si \sum_{n\in\cz^d} u_n^{2\si+2}<0, 
	$$
	it follows that $n(\cl_+)\geq 1$, so 
	$n(\cl_+)=1$. That is $\cl_+$  has exactly one negative eigenvalue. 
	\item It is possible, and this has in fact been done in \cite{Wein}, to construct the normalized waves in the $l^2$ supercritical regime $\si\geq \f{2}{d}$, but only for large enough $\la$. 
\end{itemize}
Our next result concerns the  spectral stability of the normalized waves constructed in Theorem \ref{theo:5}. 
\begin{theorem}
	\label{cor:17} 
	Let $d\geq 1, 0<\si<\f{2}{d}$ and $\la>0$. Then all normalized waves $e^{i  t c(\la) } {\mathbf u}_\la$ constructed in Theorem \ref{theo:5} satisfy $\ubb\perp Ker(\cl_+)$. As such the quantity $\cl_+^{-1} \ubb$ is uniquely defined in $Ker(\cl_+)^\perp$. In addition,   $\dpr{\cl_+^{-1} \ubb}{\ubb}\leq 0$. 
	
	Assuming the non-degeneracy condition, $\dpr{\cl_+^{-1} \ubb}{\ubb}\neq 0$, this implies $\dpr{\cl_+^{-1} \ubb}{\ubb}<0$, hence spectral stability for $e^{i  t c(\la) } {\mathbf u}_\la$. 
\end{theorem}
{\bf Remarks:} 
\begin{enumerate}
	\item The condition $\dpr{\cl_+^{-1} \ubb}{\ubb}\neq 0$ is equivalent to the lack of crossing of a pair of eigenvalues (either purely imaginary or a pair of positive and a negative one) through the origin. Crossing  is  not expected    to happen. 
	\item With the same proof, it is easy to see that the normalized waves in the $l^2$ supercritical regime $\si\geq \f{2}{d}$, $e^{i  t c(\la) } {\mathbf u}_\la$, which exist only for large $\la>0$, are also spectrally stable solutions of \eqref{10}. 
\item It is well-known~\cite{dnlsbook} that the DNLS solutions come
  into onsite and offsite symmetric varieties, i.e., centered on or
  between adjacent sites. Per our numerical computations, It is only the former that
  satisfy the condition of $n(\cl_+)=1$ and hence correspond to the
  constrained minimizers of interest herein.
        \end{enumerate}

\subsection{The existence and stability of the homogeneous waves}
Our next results concern the constrained minimization problem \eqref{40}.  We state the existence result first. 
\begin{theorem}(Existence for minimizers of homogeneous functionals)
\label{theo:10} 

Let $\om>0$. Then, the constrained minimization problem \eqref{40} has a solution ${\mathbf u}_\om$. The solution has non-negative entries only, which in addition obey \eqref{ineq1}, \eqref{ineq2}. 
The minimizer ${\mathbf u}_\om$ satisfies the Euler-Lagrange equation 
\begin{equation}
\label{EL:20} 
-\De_{disc} u_n+\om u_n - j(\om) u_n^{2\si+1} = 0,
\end{equation}
where 
$$
j(\om)=\inf\limits_{\|{\mathbf u}\|_{l^{2\si+2}}=1} J[{\mathbf u}]>0.
$$
Moreover, every minimizing sequence for \eqref{40} has a convergent subsequence, in the  $l^2$ norm,  to a constrained minimizer of \eqref{40}. 

The linearized Schr\"odinger operator 
$$
\cl_+ f_n=-\De_{disc}f_n +\om f_n    - (2\si+1) j(\om) u_n^{2\si} f_n
$$
has the property $n(\cl_+)=1$. 
\end{theorem}
The minimizers constructed in Theorem \ref{theo:10} do not satisfy the required profile equation \eqref{20}, but rather the Euler-Lagrange equation \eqref{EL:20}. It is clear however, that we can construct solutions of \eqref{20} starting from solutions $\ubb$ as in Theorem \ref{theo:10}. Indeed, setting $\vp_n:=\left(j(\om)\right)^{\f{1}{2\si}} u_n$ produces solutions of \eqref{20}. 

Our next result concerns the stability of these waves. Unfortunately, we need to put forward some technical constructions before we state a rigorous result.
\begin{definition}
	\label{defi:29} Let $\om>0$. We say that $\ubb_\om$ is a limit wave, if there exists a sequence $\de_j\to 0$, so that $\ubb_{\om+\de_j}$ are minimizers for \eqref{40}  and $\lim_j \|\ubb_{\om+\de_j}-\ubb_\om\|_{l^2}=0$. 
\end{definition} 
Since $\{\ubb_{\om+\de_j}\}_j$ is a minimizing sequence for \eqref{40},  and since by the results of Theorem \ref{theo:10}, we can select a further $l^2$ convergent subsequence, it is clear that  limit waves are constrained minimizers of \eqref{40}. As such, they do  satisfy the Euler-Lagrange equation \eqref{EL:20}. Starting with ${\mathbf u}$,  we can construct waves of the original DNLS via  $\vp_n:=\left(j(\om)\right)^{\f{1}{2\si}} u_n$. We call these also limit waves. 

Our next result concerns the stability of these waves. 
\begin{theorem}
	\label{theo:15} 
	Let $\om>0$ and assume $\vp_\om$ is a limit wave, in the sense of Definition \ref{defi:29}. Assume also that $\vp_\om$ is non-degenerate, in the sense that 
	$
	Ker(-\De_{disc}+\om - (2\si+1) \vp^{2\si})=\{0\}.
	$
	
	 Then, the solution $e^{i \om t}\vp_\om$ is spectrally stable in the context of DNLS, if 
	 \begin{equation}
	 \label{300}
	  \lim_{\de_j\to 0} \f{\|\vp_{\om+\de_j}\|_{l^2}^2 -\|\vp_{\om}\|_{l^2}^2}{\de_j}\geq 0.
	 \end{equation}
	 and it is unstable if 
	 $$
	  \lim_{\de_j\to 0}  \f{\|\vp_{\om+\de_j}\|_{l^2}^2 -\|\vp_{\om}\|_{l^2}^2}{\de_j}<0.
	 $$	 
\end{theorem}
{\bf Remarks:} 
\begin{enumerate}
	\item Under the assumptions made in Theorem \ref{theo:15},  the limit  
	$$
	\lim_{\de_j\to 0} \f{\|\vp_{\om+\de_j}\|_{l^2}^2 -\|\vp_{\om}\|_{l^2}^2}{\de_j}
	$$
	always exists.  
	\item The non-degeneracy  property of the wave, 
	$$
	Ker(-\De_{disc}+\om - (2\si+1) {\mathbf \vp}^{2\si})=\{0\},
	$$
	 is a technical, but useful statement, which is expected to hold, at least for a generic subset of 
	 the parameter $\om$. 
\end{enumerate}
We have the following  corollary. 
\begin{corollary}
	Assume that for an interval $(a,b)\subset \rone_+$, we have that $\om\to \vp_\om$ is continuous in the $l^2$ norm and all the waves $\vp_\om$ are non-degenerate. Then, the scalar function $\om\to \|\vp_\om\|_{l^2}$ is differentiable, and the wave 
	$e^{\i \om t} \vp$ is stable \underline{if and only if}  
	$$
	\p_\om  \|\vp_\om\|_{l^2}^2\geq 0.
	$$
	\label{cor1}
\end{corollary}
{\bf Remark:} This result is of course identical of the Grillakis-Shatah-Strauss stability condition for stability of waves in continuous models. This is, as far as we know, the first such rigorous result in the setting of discrete NLS. 

\begin{figure}[t]
\includegraphics[scale=0.3]{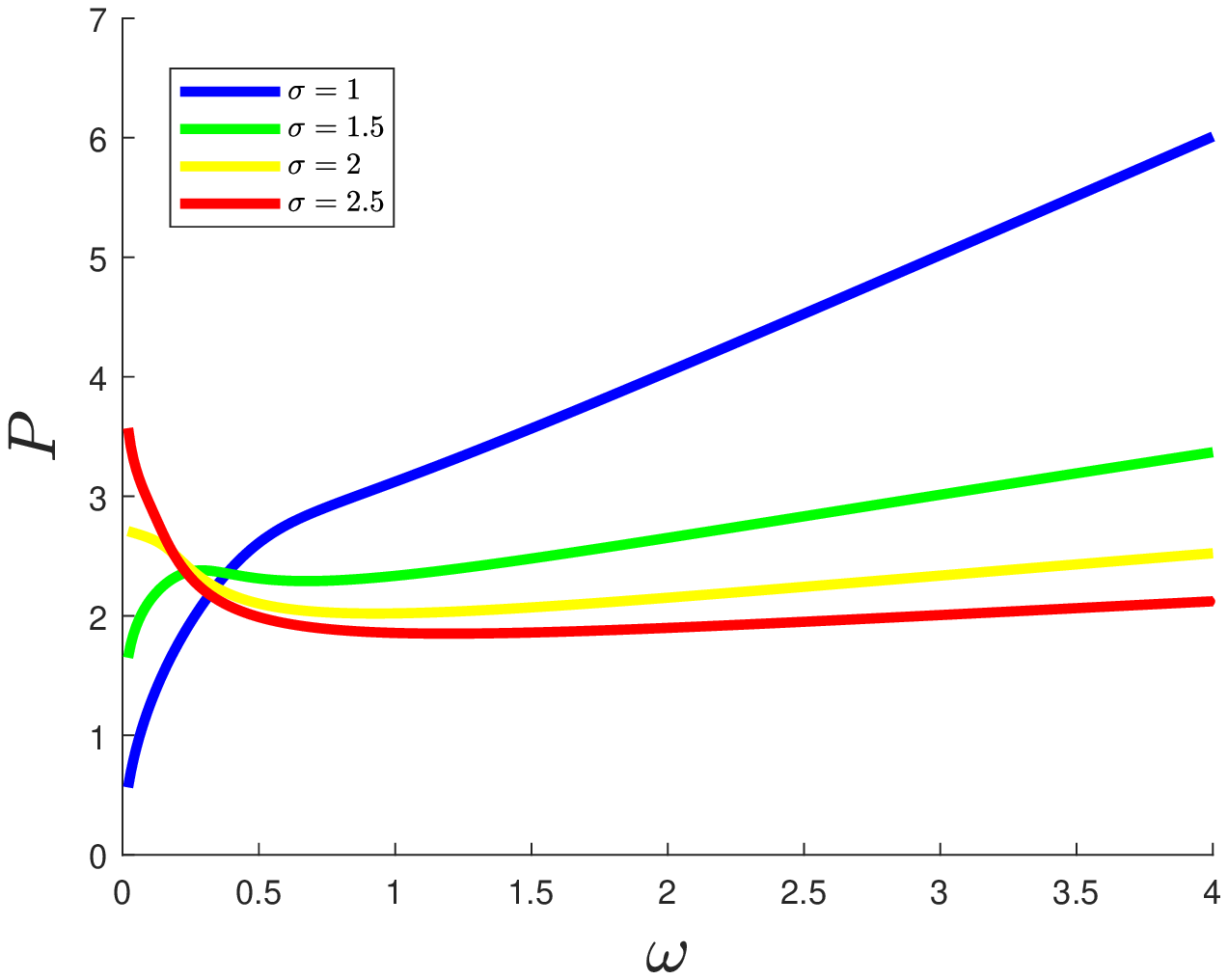}
\includegraphics[scale=0.3]{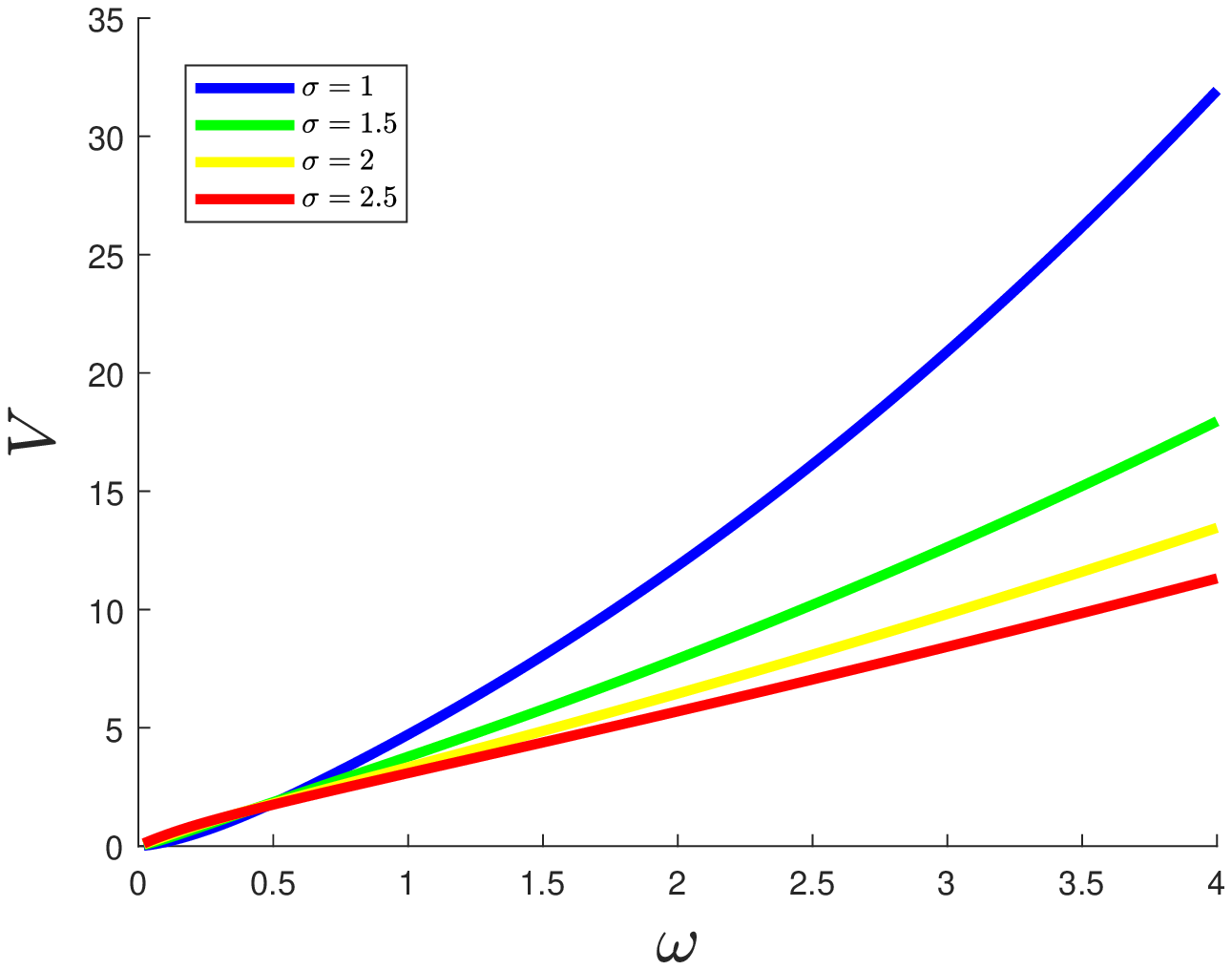}
\\
\includegraphics[scale=0.3]{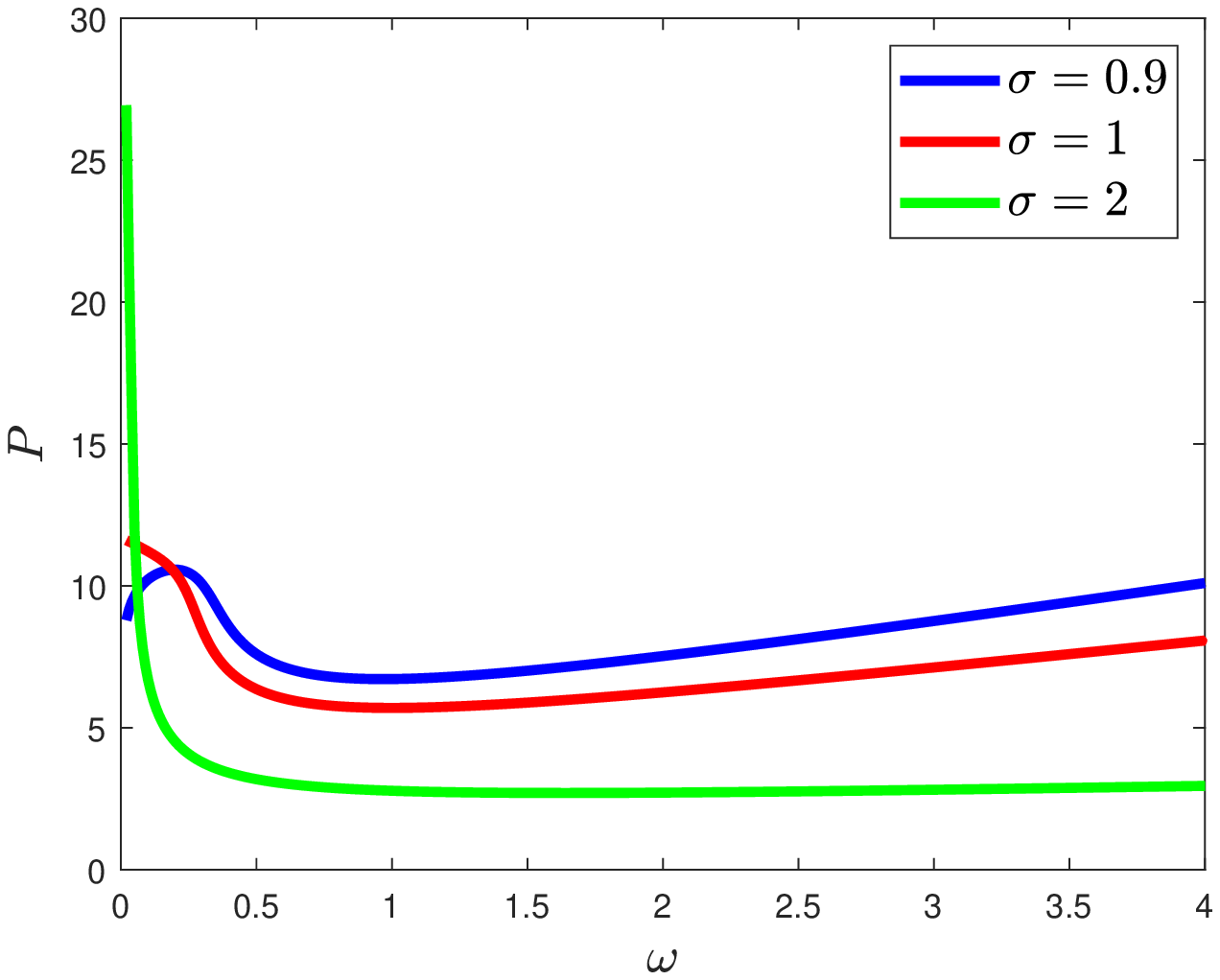}
\includegraphics[scale=0.3]{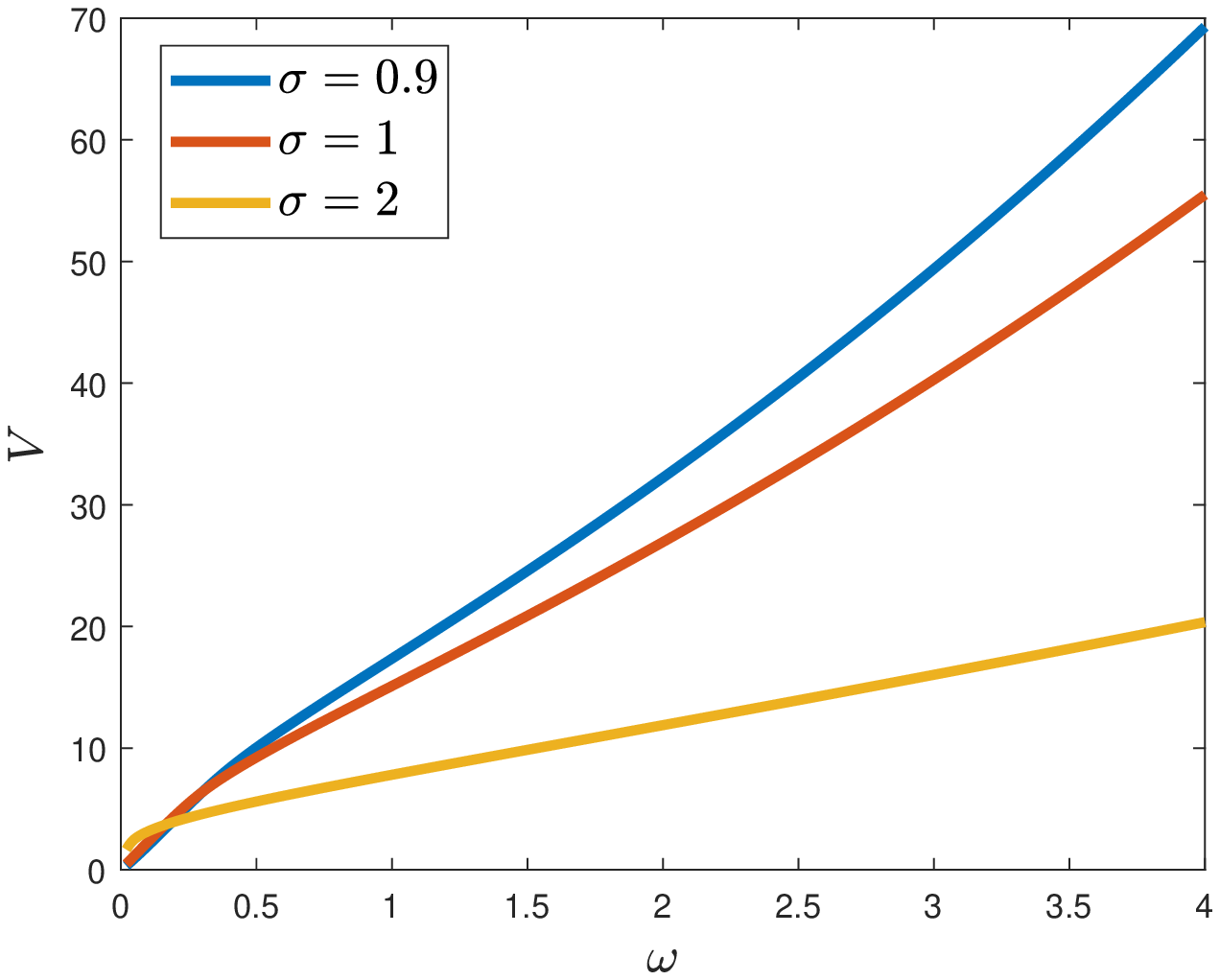}
\caption{The quantities $P$ and $V$ as functions of the frequency $\om$ in dimension $d=1$ (top) and $d=2$ (bottom).}
\label{PTV_vs_om_d=1}
\end{figure}


\begin{figure}[t]
\includegraphics[scale=0.45]{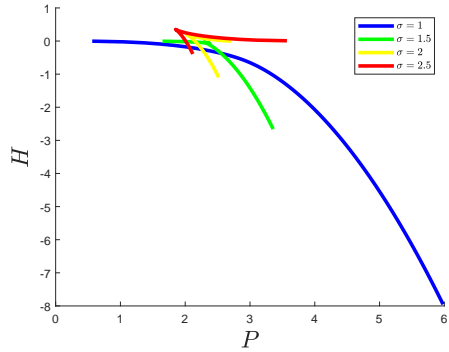}
\caption{The functions $H[\vp_\om]$ vs. $P[\vp_\om]$ in $d=1$ for
  different values of $\sigma$.}
\label{fig-HvsP}
\end{figure}


Next, we express everything in terms of the function $j(\om)$. 
\begin{proposition}
	\label{prop:908} 
	The function $\om\to j(\om)$ is concave down. As such, it has a second derivative a.e. on $(0, \infty)$. Moreover, at such points of differentiability,  
	\begin{eqnarray}
	\label{335} 
	\|\vp_\om\|_{l^2}^2  = j'(\om) j^{\f{1}{\si}}(\om).
	\end{eqnarray}
	Thus, the stability condition is exactly that $j^{1+\f{1}{\si}}(\om)$ is convex or 
	\begin{equation}
	\label{350}
		s(\om) := j''(\om)+ \f{(j'(\om))^2}{\si j(\om)}>0.
	\end{equation} 
\end{proposition}

Figures \ref{PTV_vs_om_d=1}--\ref{fig-HvsP}
show the diagrams
of $P$ and of $V$, as well as the Hamiltonian $H$ vs. $P$ for
different values of $\sigma$. It is important to highlight that in
line with
the earlier works of, e.g.,~\cite{malomed,mal2}, there exist values of
$\sigma$
(e.g., for $1.34<\sigma<2$ for $d=1$~\cite{mal2}) for which there is
multi-stability and the graph of $P$ vs. $\omega$ is non-monotonic;
see also Figure \ref{PTV_vs_om_d=1} for a case with $d=2$ (bottom
panels).

Figures \ref{fig-stability-1d-sigma=1}--\ref{fig-stability-1d-sigma=2.5} (left panels) show $s(\om)$ compared to the quantity $\partial_\om \|\vp_\om\|_{\ell^2}^2$ from Corollary \ref{cor1}. According to Proposition \ref{prop:908} and Corollary \ref{cor1}, both functions can be used as indicators of spectral stability. We observe their zero crossings are identical and correspond to stability switches at the power extrema displayed in Figure \ref{PTV_vs_om_d=1} (top panels).

It is relevant to note here that for ease of visualization in these
figures, we multiply $s(\om)$ by a constant pre-factor so that it is
on the same order as $\partial_\om \|\vp_\om\|_{\ell^2}^2$. Of course,
this has no effect on the zero crossings and the overall sign which
are central for our stability conclusions.

To corroborate the analytical results further, the middle and right
panels of
Figs. \ref{fig-stability-1d-sigma=1}--\ref{fig-stability-1d-sigma=2.5}
illustrate stability switching via numerical computation of spectral
stability eigenvalues in Eq.~(\ref{200}), for fixed $\om$ in the
regions between the zero crossings of $s(\om)$. For $\sigma = 2, 2.5$,
the zero crossings occur approximately at $\om = 1, 1.2$, so we check
the stability at $\om=0.5$ and at $\om = 1.5$. Our results predict
that in one of these cases the wave should be stable, while in the
other it should be unstable. For $\sigma=1.5$, $s(\om)$ features two
zero crossings, near $\om=0.4$ and $\om=0.8$. Hence it is also
relevant to test a value to the left of $\om=0.4$, one between the two
crossings and one to the right of the rightmost crossing; we choose
$\om=0.1$,
$0.5$ and $1.5$, respectively.
For consistency, we display the $\sigma=2$ and $2.5$ cases
 for the same test values $\om = 0.5, 1.5$, observing that in the
 latter
 too, a zero-crossing occurs. This is, respectively, at $\omega=1$
 and $\omega=1.2$ for these two cases. In all the cases, our
 stability conclusions are in line with the standard VK criterion
 and the newly proposed criterion of Eq.~(\ref{350}).

 Moreover, our numerical computations lead us to
   conjecture here the following conclusion: for the critical index
   $\si=\f{2}{d}$, the excitation threshold is achieved precisely at $\omega=1$
   and moreover $\vp_\om$, for $\om>1$ are stable.
   That is, 
   $\min_{\om>0} \|\vp_\om\|^2=\|\vp_1\|^2$, $\p_\om  \|\vp_\om\|^2>0$
   for $\om>1$. The opposite monotonicity is true leading to
   instability for $\om <1$.

\begin{figure}[htb]
\includegraphics[scale=0.33]{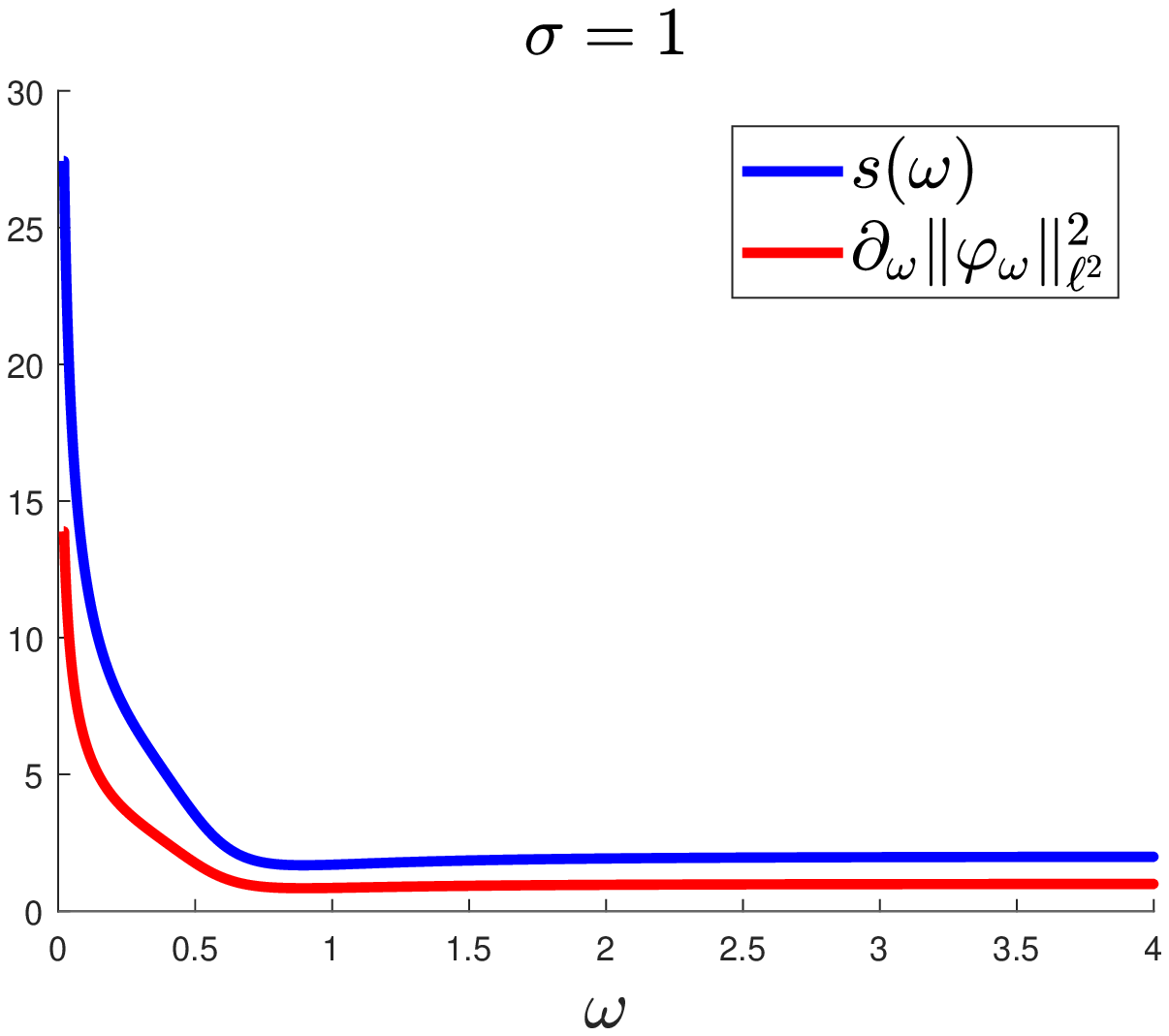}
\includegraphics[scale=0.33]{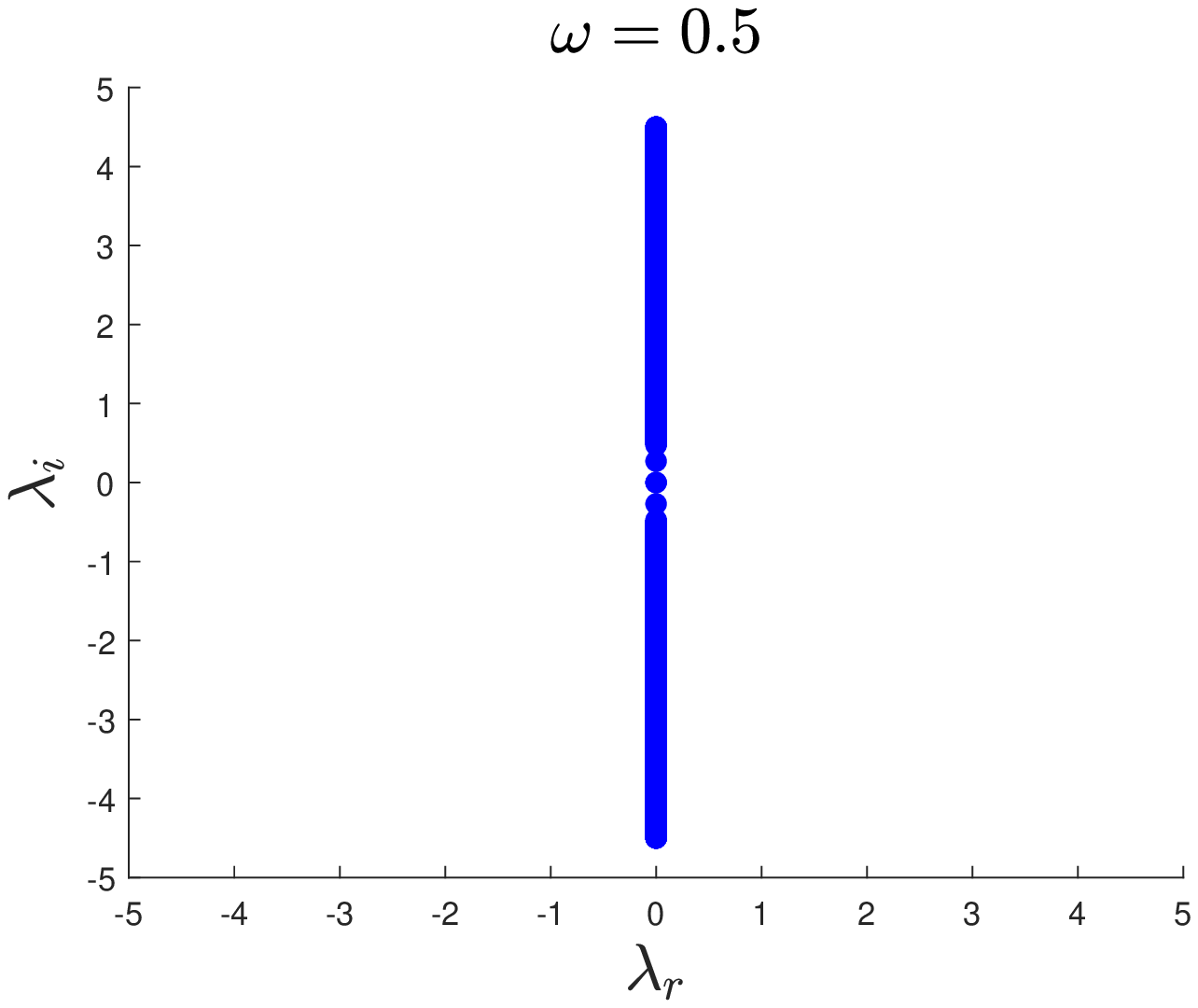}
\includegraphics[scale=0.33]{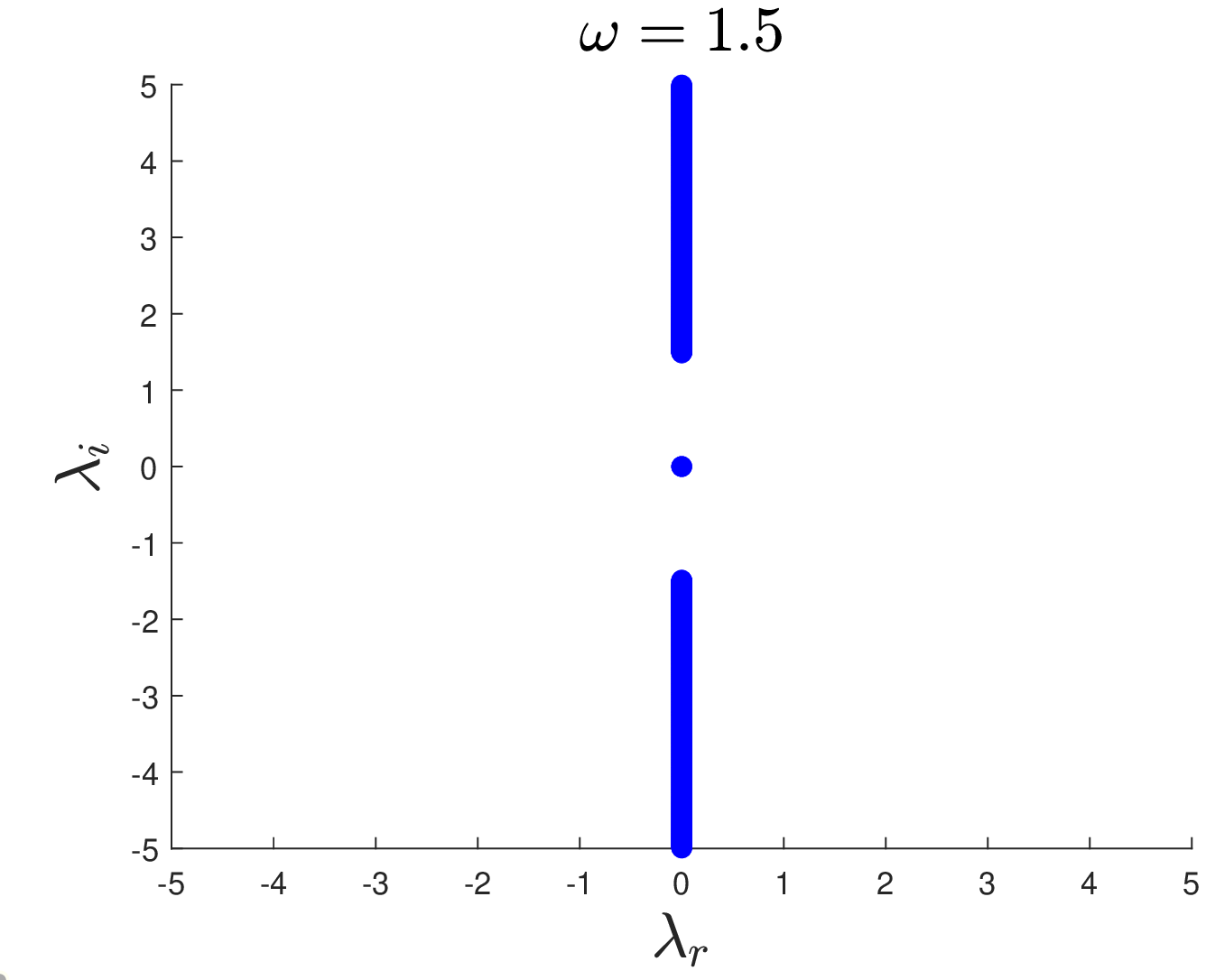}

\caption{The functions $s(\om)$ and $\partial_\om \|\vp_\om\|_{\ell^2}^2$ vs. $\om$ in one dimension, for $\sigma=1$ (left). The middle and right panels show the spectral plane $(\la_r,\la_i)$ for eigenvalues $\la = \la_r + i\la_i$ of Eq. \eqref{200}, for $\om=0.5$ (middle) and $\om=1.5$ (right). As expected, both waves are spectrally stable.}
\label{fig-stability-1d-sigma=1}
\end{figure}

\begin{figure}[htb]
\includegraphics[scale=0.24]{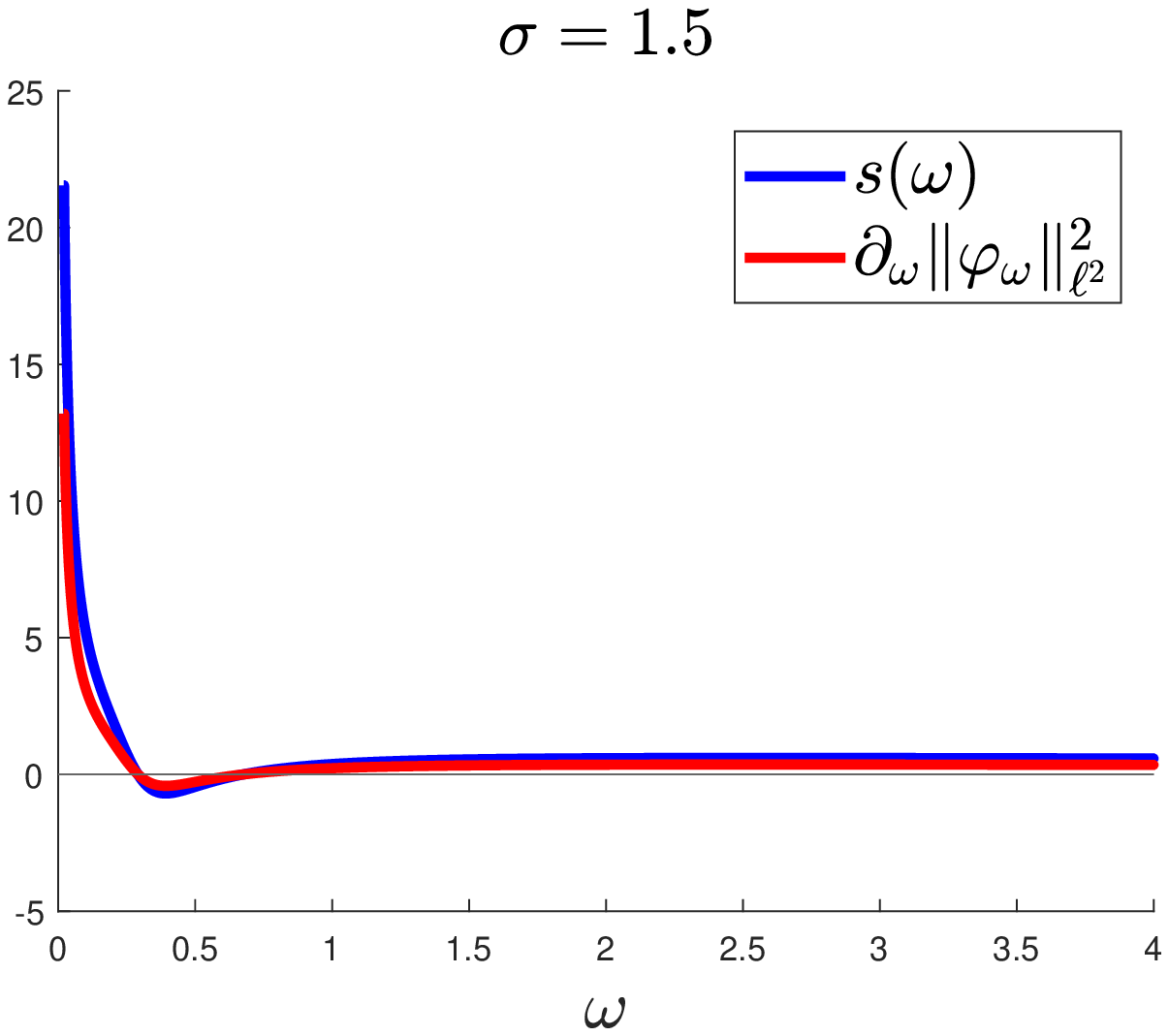}
\includegraphics[scale=0.24]{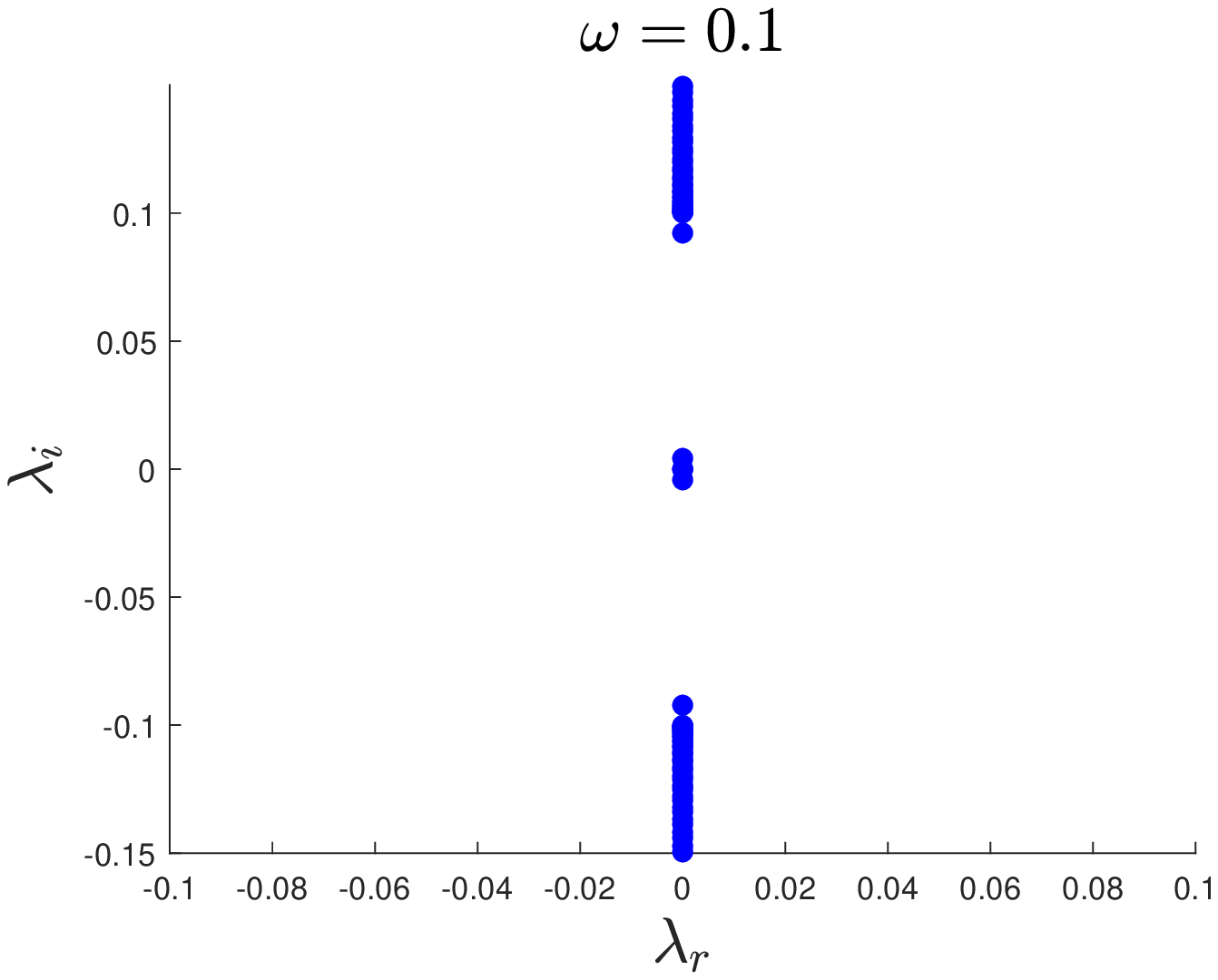}
\includegraphics[scale=0.24]{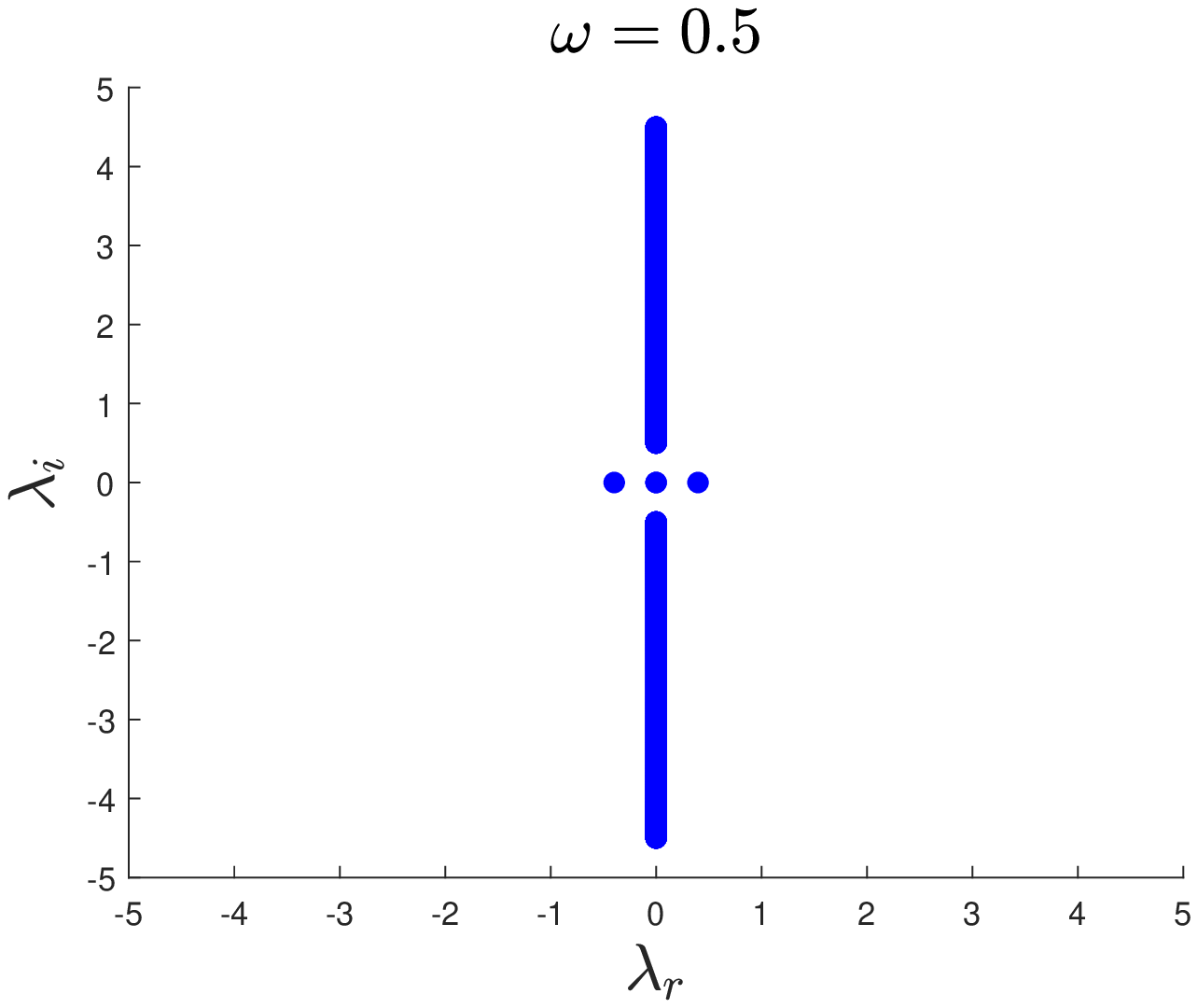}
\includegraphics[scale=0.24]{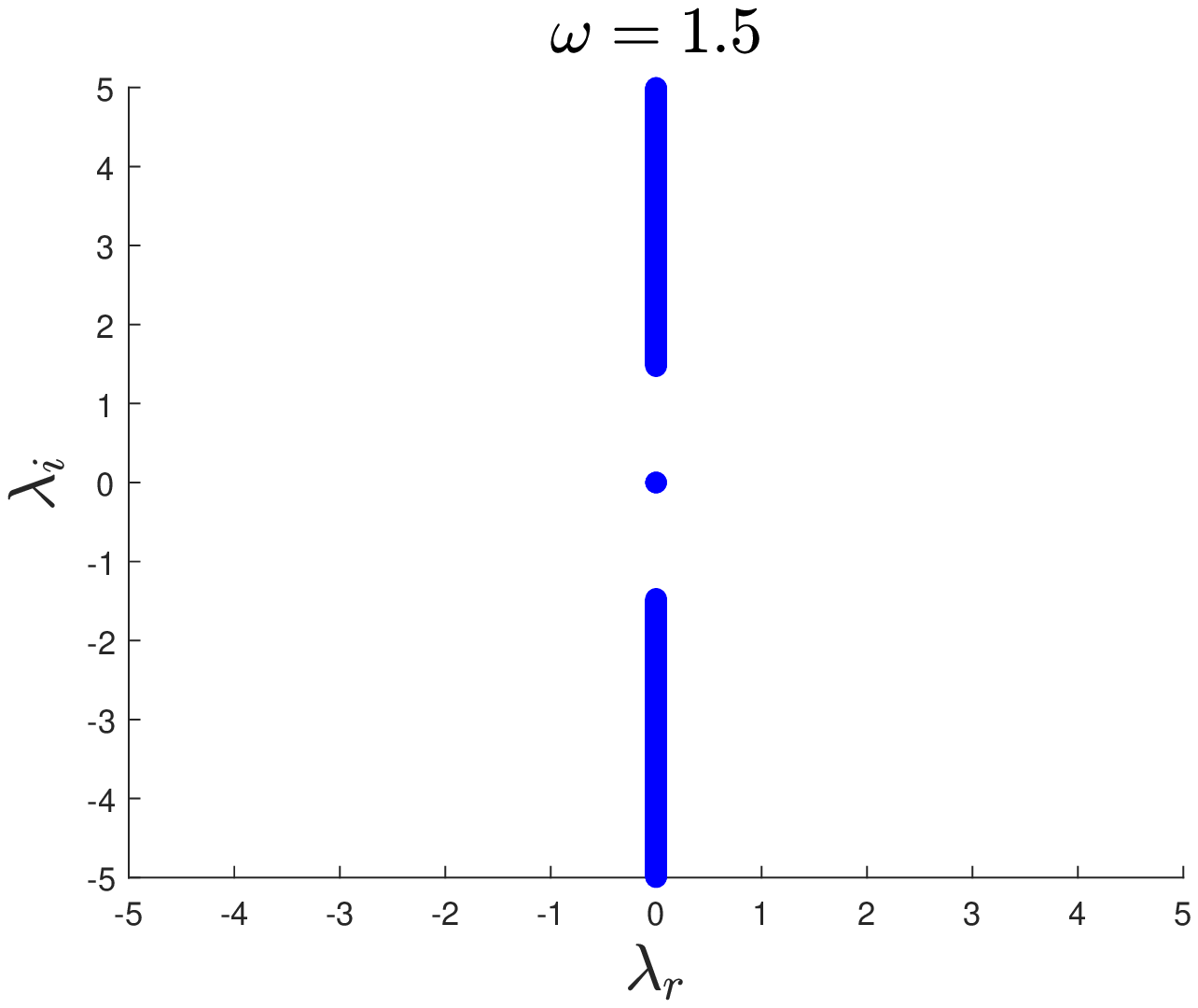}

\caption{Similar to Figure \ref{fig-stability-1d-sigma=1} but now for
  $\sigma = 1.5$. The zero crossings of $s(\om)$ occur near $\om
  \approx 0.4$ and $\om \approx 0.8$.}
\label{fig-stability-1d-sigma=1.5}
\end{figure}

\begin{figure}[htb]
\includegraphics[scale=0.33]{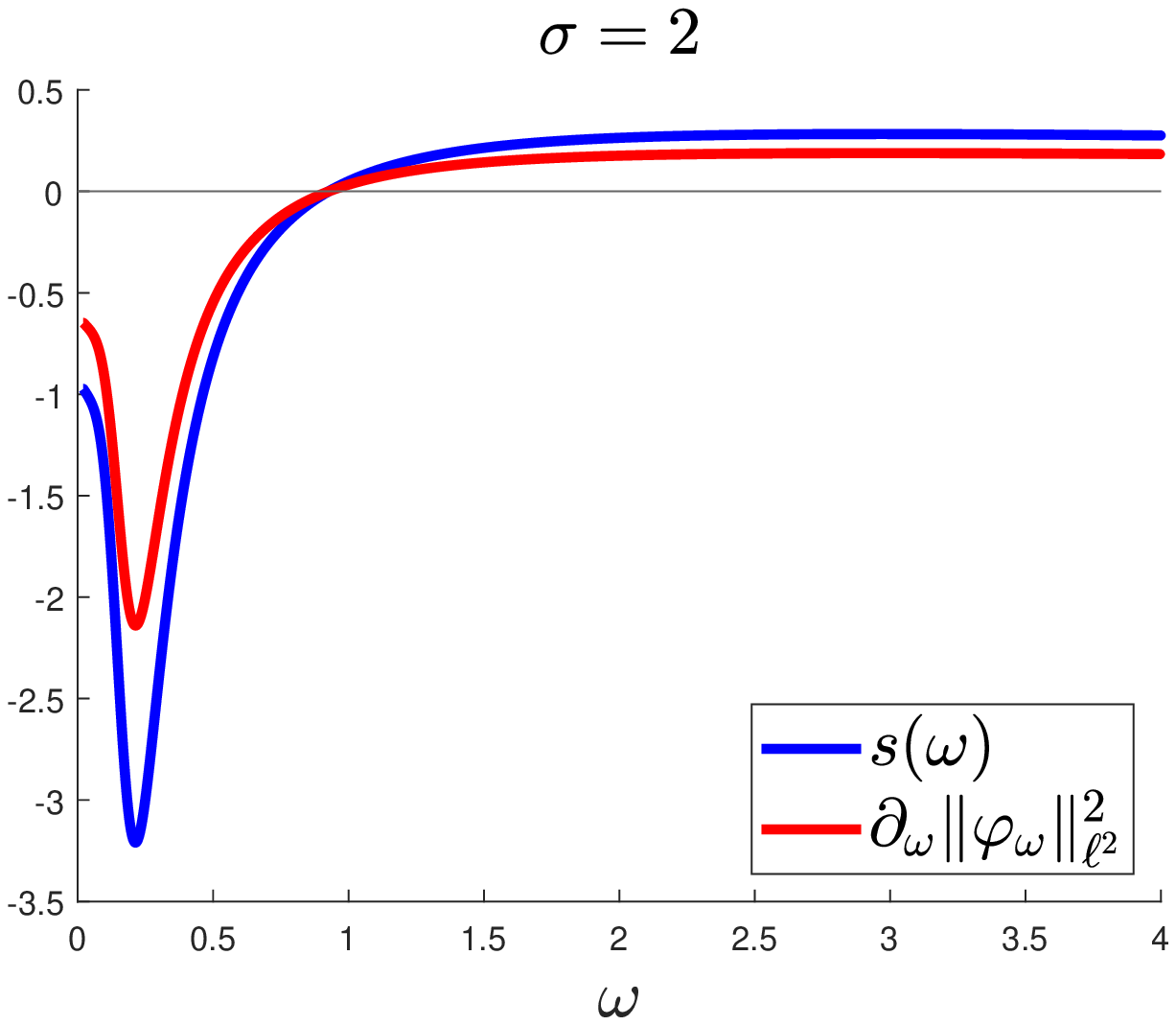}
\includegraphics[scale=0.33]{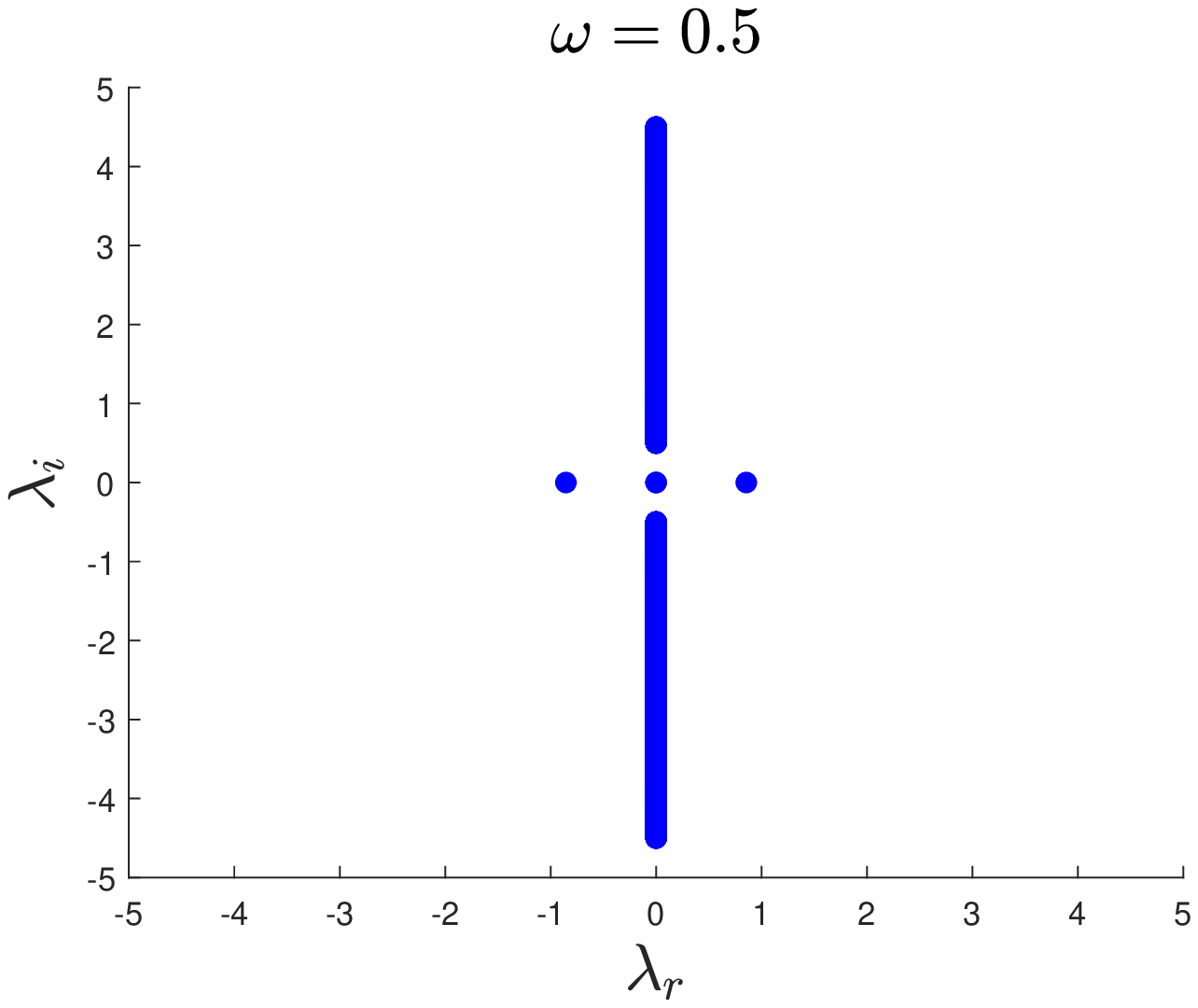}
\includegraphics[scale=0.33]{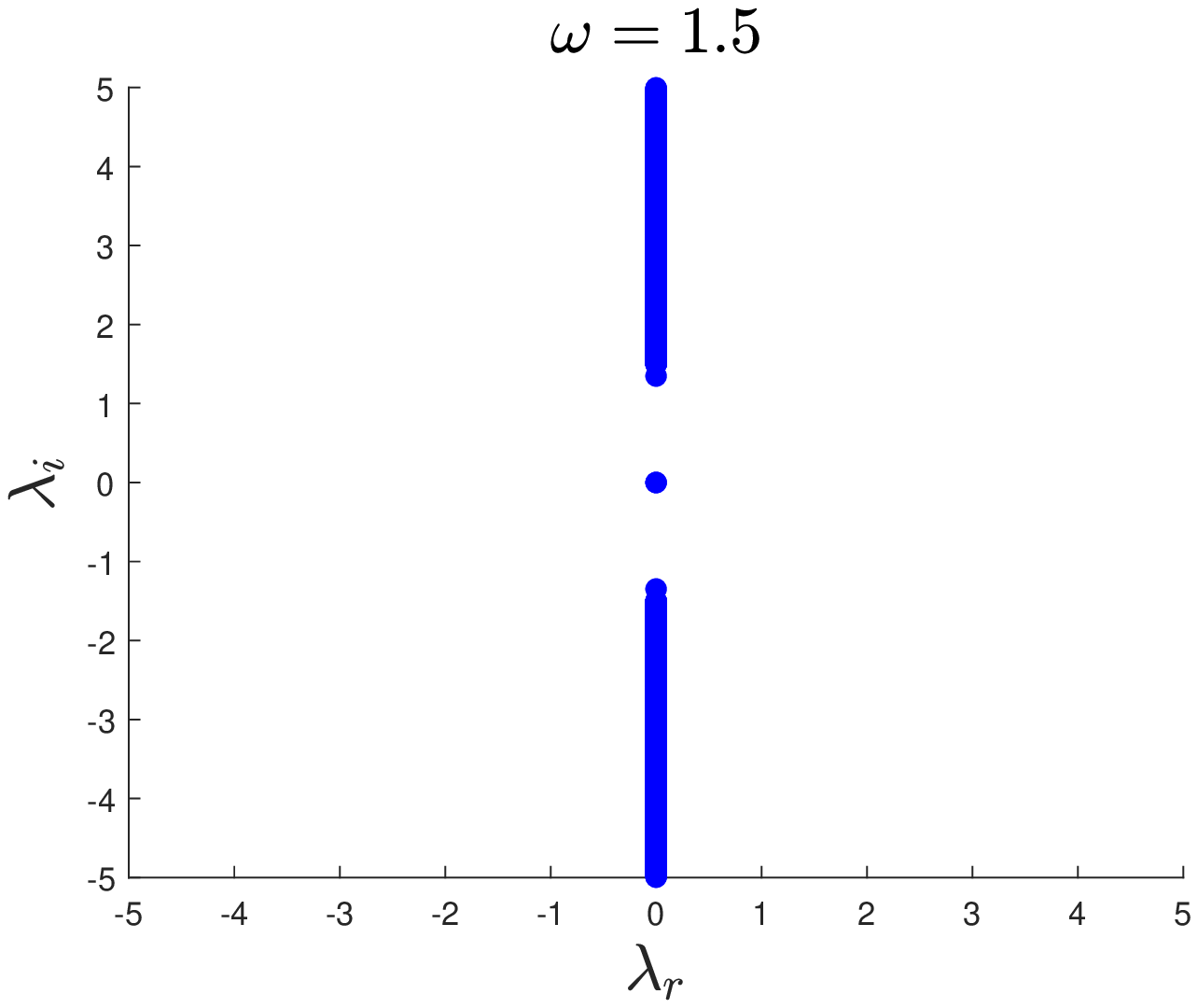}

\caption{Similar to Figure \ref{fig-stability-1d-sigma=1}, but now for $\sigma = 2$. The zero crossing occurs at $\om=1$. }
\label{fig-stability-1d-sigma=2}
\end{figure}

\begin{figure}[htb]
\includegraphics[scale=0.33]{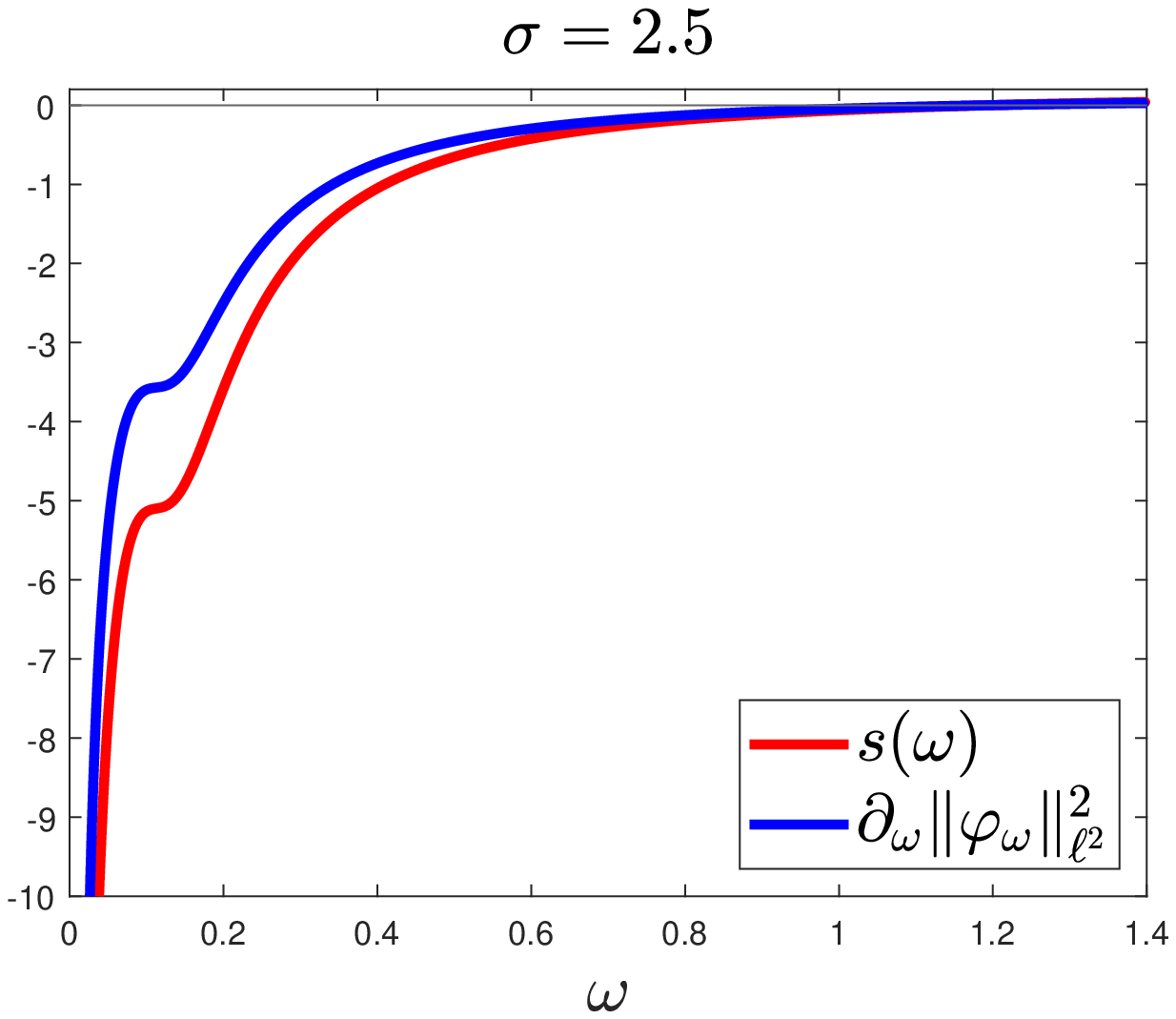}
\includegraphics[scale=0.33]{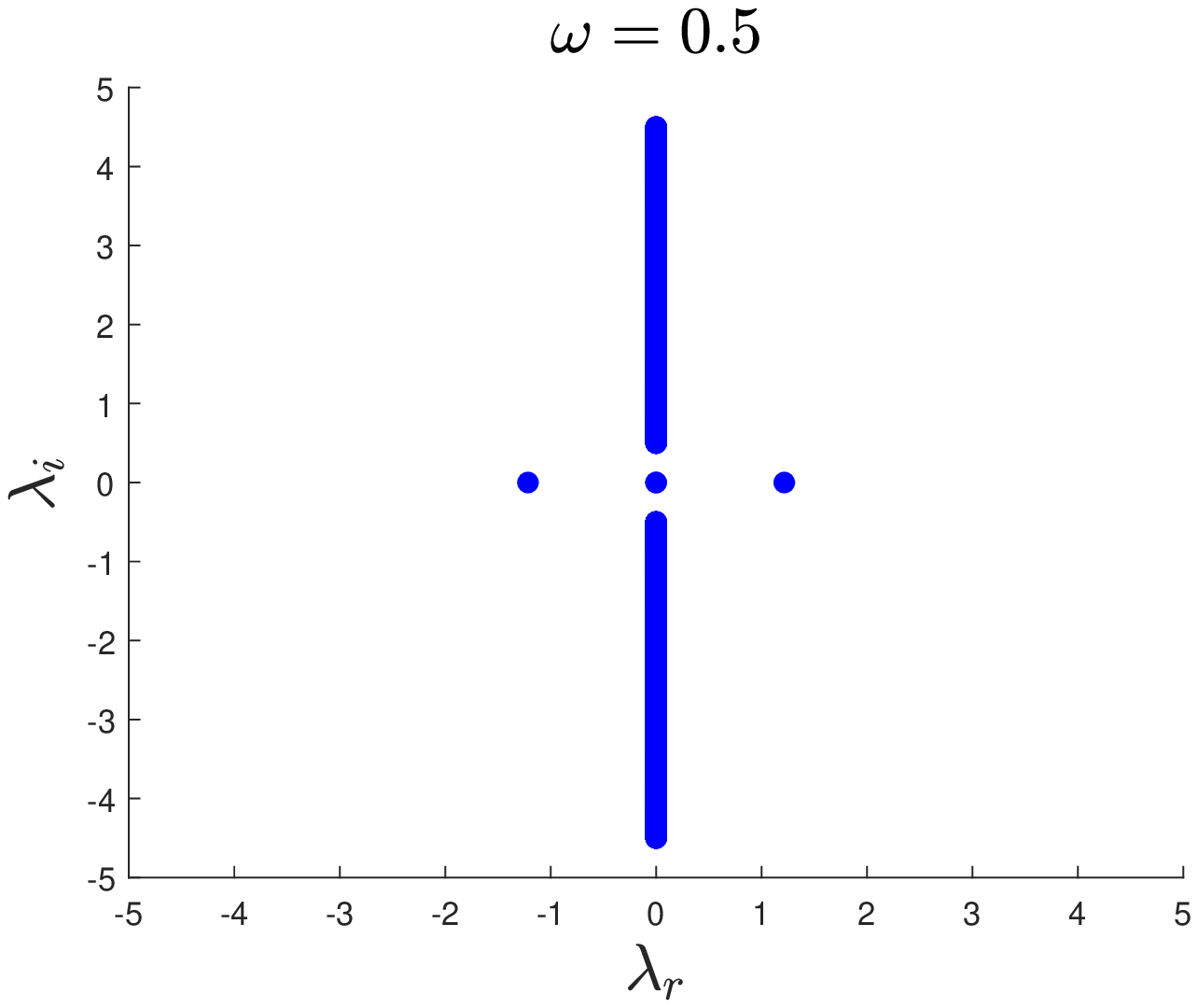}
\includegraphics[scale=0.33]{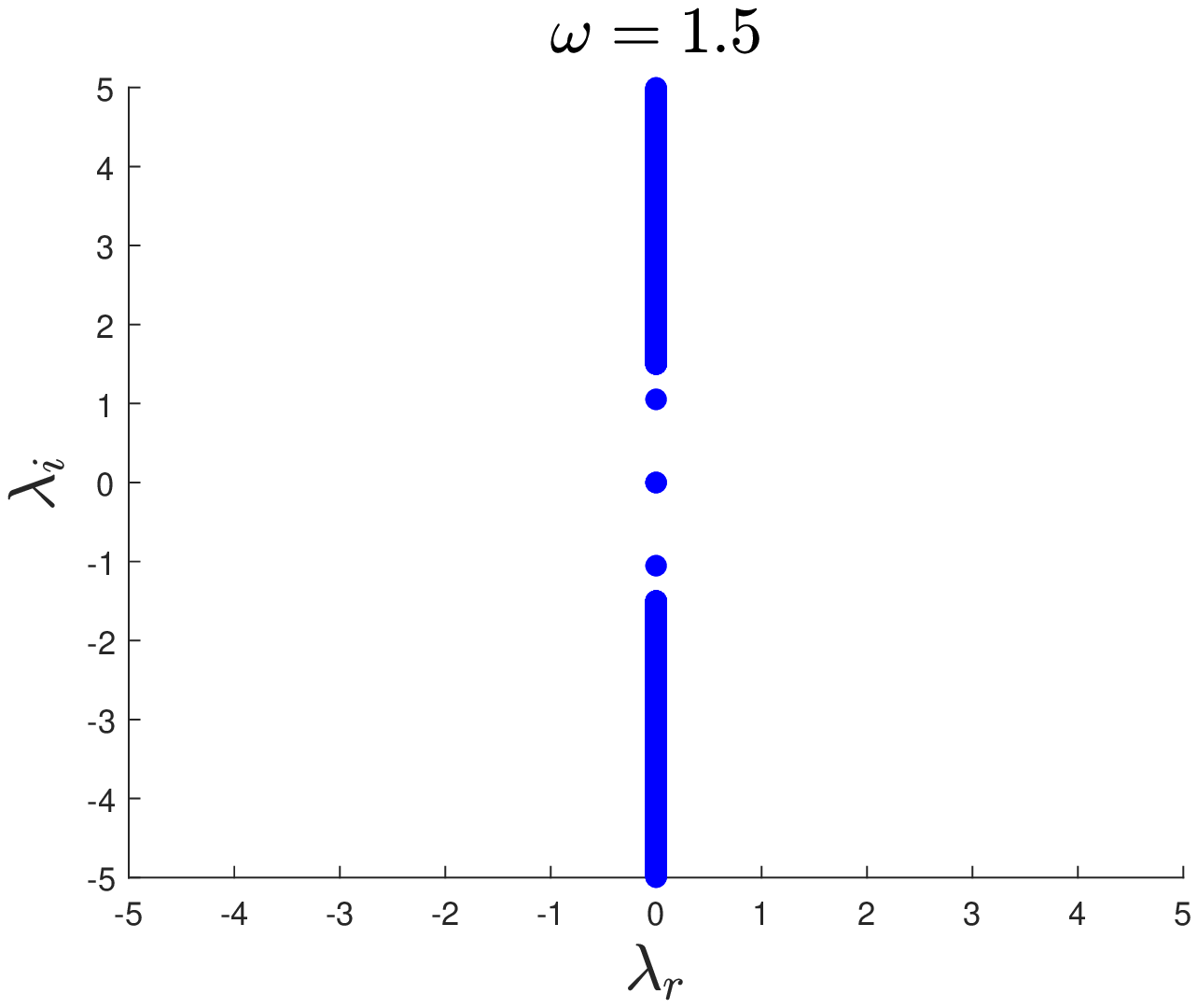}
\caption{Similar to Figure \ref{fig-stability-1d-sigma=1}, but now for $\sigma = 2.5$. The zero crossing occurs near $\om = 1.2$.}
\label{fig-stability-1d-sigma=2.5}
\end{figure}

\section {Preliminaries} 
We consider  the following spaces 
$$
l^p=\{\{x_n\}_{n=-\infty}^\infty: \|x\|_{l^p}=\left(\sum_{n=-\infty}^\infty |x_n|^p\right)^{1/p}<\infty \}.
$$   

In particular the spaces $l^2$ can be identified via the isometry map $\cf: l^2\to L^2[0,1]$  as follows  
$$
x(\xi)=\cf[\{x_n\}]:=\sum_n x_n e^{2\pi i n \xi},\ \ x_n=\int_0^1 x(\xi) e^{-2\pi i n \xi} d\xi.
$$
Note that $\|x\|_{L^2[0,1]}=\|\{x_n\}\|_{l^2}$.  Sometimes, we denote ${\mathbf u}=\{u_n\}_{n=-\infty}^\infty$ and ${\mathbf u}(\xi):=\sum_n u_n e^{2\pi i n \xi}$. We will often  tacitly identify the sequence ${\mathbf u}$ and the function ${\mathbf u}(\xi)$. 
Note the basis vectors $e_n$, which are defined via the Kronecker
$\delta$'s:
$
e_n(m)=\left\{
\begin{array}{cc}
1 & n=m \\
0 & n\neq m
\end{array}.
\right.
$
\subsection{The discrete Laplacian}
\label{sec:2.1} 
 The one dimensional discrete Laplacian is given explicitly by 
 $$
 (\De_{disc.} {\mathbf u})_n=u_{n+1}-2u_n+2 u_{n-1}.
 $$
 In fact, $\De_{disc}$ has a nice representation in terms of the shift operators. Indeed, let $S:l^2\to l^2$, be defined by $S[{\mathbf u}](n)=u_{n+1}$, while its inverse/adjoint 
 $S^*=S^{-1}$, is given by $S^{-1}[{\mathbf u}](n)=u_{n-1}$. In such a case,  we can represent $ \De_{disc}= S+S^{-1}- 2 Id$. 
  On the level of $L^2[0,1]$ functions, we have 
  \begin{eqnarray*}
  	u_{n+1}-2u_n+2 u_{n-1} &=&  \int_0^1 \ubb(\xi) [e^{-2\pi i (n+1) \xi} +e^{-2\pi i (n-1) \xi}-2 e^{-2\pi i n \xi}]d\xi= \\
  	&=& -4 \int_0^1 \ubb(\xi)\sin^2(\pi \xi) e^{-2\pi i n \xi} d\xi. 
  \end{eqnarray*} 
  That is, on the set of $L^2[0,1]$ functions, the discrete Laplacian can be realized as the  Fourier multiplier $-4\sin^2(\pi \xi)$. 
  
  More generally, for reasonable functions $f(-\De_{disc})$ acts as follows $f(-\De_{disc}) \ubb$ corresponds to a function $\ubb(\xi)f(4 \sin^2(\pi \xi))$. For example, one can easily see (summation by parts) that 
  $$
  \dpr{-\De_{disc} \ubb}{\ubb}=\sum_n (2u_n-u_{n+1}- u_{n-1})\bar{u}_n = \sum_n |u_{n+1}-u_n|^2\geq 0,
  $$
  whence $-\De_{disc}$ is a positive operator. Equivalently, one could have computed this on the level of functions $\ubb(\xi)$ as follows 
  $$
  \dpr{-\De_{disc} \ubb}{\ubb}=4 \int_0^1 |\ubb(\xi)|^2 \sin^2(\pi\xi) d\xi. 
  $$
  In particular, since 
  $
  \int_0^1 |\ubb(\xi)|^2 \sin^2(\pi\xi) d\xi\leq \|\ubb\|^2, 
  $
  we conclude that 
  $$
  \dpr{-\De_{disc} \ubb}{\ubb}\leq 4 \|\ubb\|^2,
  $$
  whence $0<-\De_{disc}\leq 4$ is a bounded operator, with norm at most\footnote{one can easily see that the norm is actually exactly four, by taking test functions $\ubb$ supported near $\xi\sim \f{1}{2}$} $4$. 
  In higher dimensions, $d\geq 1$, one might reach   similar conclusions about $\De_{disc}$, for example $0<-\De_{disc}\leq 4d$.

 \subsection{The heat semigroup $e^{t\De_{disc}}$ and the Perron-Frobenius property}
 
 We  introduce  the semigroup generated by $\De_{disc}$   via the heat equation on $l^2$, namely 
 \begin{equation}
 \label{101}
 \left\{
 \begin{array}{l}
 \p_t u_n(t)= u_{n+1}(t)-2u_n + u_{n-1}(t)=(\De_{disc}\ubb)_n. \\
 u_n(0)=u_n
 \end{array}
 \right.
 \end{equation}
 Equivalently, the solution is given by $\ubb(t)$, given by $\ubb(0, \xi)e^{-4t\sin^2(\pi \xi)}$. More explicitly 
 \begin{equation}
 \label{201}
 u_n(t)=\sum_{m} u_m \int_0^1 e^{-4t\sin^2(\pi \xi)} e^{-2\pi  i (n-m) \xi} d\xi.
 \end{equation}
 Introduce the coefficients $K_n(t)=\int_0^1 e^{-4t\sin^2(\pi \xi)} e^{-2\pi  i n \xi} d\xi$ of the heat semigroup. In other words, 
 $$
 \ubb(t)=e^{t \De_{disc}}\ubb(0)=\sum_m K_{n-m}(t) u_m(0)={\mathbf K(t)}*\ubb(0).
 $$
 \begin{proposition}
 	\label{prop:10}
 	For every $t>0$, the sequence  $\{K_n(t)\}$ is even, positive and bell-shaped. That is, for every $n\geq 1$,  $K_n(t)=K_{-n}(t)>0$. For every $t\in (0,1]$, ${\mathbf K}$ is bell-shaped as well, that is  $K_0(t)> K_1(t)> \ldots $. 
 \end{proposition}
 \begin{proof}
 	First, we see that if $\{u_n(0)\}_n$ is real, then $\{u_n(t)\}_n$ is real, which implies that $K_n(t)$ is real as well. Thus, 
 	$$
 	K_n(t)=\int_0^1 e^{-4t\sin^2(\pi \xi)} \f{e^{-2\pi  i n \xi} + e^{2\pi  i n \xi}}{2} d\xi = \int_0^1 e^{-4t\sin^2(\pi \xi)}  \cos(\pi n \xi) d\xi. 
 	$$
 	From this formula, it is clear that $K(n)=K(-n)$.  

 	We now show that $K(j)\geq K(j+1), j=0, 1, \ldots$.  
 	
 	For this step, we use the following explicit formula, which is obtained by the Taylor expansion of the exponential function. We have 
 	\begin{eqnarray*}
 		e^{2t} e^{ t \De_{disc}}&=&  e^{t(S+S^{-1})}= \sum_{n=0}^\infty  \f{t^n(S+S^{-1})^n}{n!} =  \sum_{n=0}^\infty \f{t^n}{n!} \sum_{k=0}^n S^{n-2k} 
 		\left(\begin{array}{c} n \\ k \end{array} \right) = \\
 		&=& \sum_{j=0}^\infty t^j (S^j+S^{-j}) \sum_{k=0}^\infty \f{t^{2k}}{k! (k+j)!}.
 	\end{eqnarray*}
 	From this last formula, it follows that for $j\geq 0$, 
 	\begin{equation}
 		\label{par:7} 
 		K_{-j}(t)=K_j(t)= e^{-2t} t^j \sum_{k=0}^\infty \f{t^{2k}}{k! (k+j)!}. 
 	\end{equation}
 	Clearly, at least for $t\in (0,1]$, $j\geq 0$, we have $K_j(t)> K_{j+1}(t)$. 
 \end{proof}
 We now turn our attention to a general (linear)  Schr\"odinger operator in the form 
 $$
 (\cl f)_n=-\De_{disc} f_n - V_n f_n=\sum_{j\in\cz^d: |j-n|=1} (f_{j}-f_n) - V_n f_n,
 $$
 for bounded potentials ${\mathbb V}=\{V_n\}_{n\in \cz^d}$.  From the general form of Weyl's criteria (see the proof of Theorem \ref{prop:PF} in the Appendix), we show that  $\si_{ess}(\cl)=\si_{ess}(-\De_{disc})= \si(-\De_{disc})=[0, 4d]$. Our interest is in the lowest eigenvalue of $\cl$, if such an object exists. Assume that it does - by the min-max principle, we   assume that  
 $$
 \la_0(\cl)=\inf_{\|{\mathbf f}\|=1} \dpr{\cl {\mathbf f}}{{\mathbf f}}=\inf _{\|{\mathbf f}\|=1} \sum_{n\in \cz^d} \left(\sum_{j\in\cz^d: |j-n|=1} |f_{j}-f_n|^2 - V_n f_n^2\right)<0.
 $$
 In the continuous case, it is well-known that such an eigenvalue is
 simple, with pointwise positive eigenfunction. This is known  as the
 Perron-Frobenius theorem and so we refer to it as the
 Perron-Frobenius property. We have the same result for discrete
 Schr\"odinger operators, which to the best of our knowledge  appears to be a new result.

 \begin{theorem}(Discrete Schr\"odinger operators are Perron-Frobenius) 
 	\label{prop:PF} 
 	
 	Let $d\geq 1$ and ${\mathbb V}\in l^\infty(\cz^d)$. Assume that $\inf_{\|{\mathbf f}\|=1} \dpr{\cl {\mathbf f}}{{\mathbf f}}<0$. Then, $ \la_0(\cl)=\inf_{\|{\mathbf f}\|=1} \dpr{\cl {\mathbf f}}{{\mathbf f}}$ is a simple eigenvalue for $\cl$, with an eigenfunction (ground state) ${\mathbf g}: g_n\geq 0$.  In particular, there exists $\de>0$, so that $\cl_{\{{{\mathbf g}\}^\perp}}\geq \la_0(\cl)+\de$.
 \end{theorem}
We provide the proof in the Appendix, as it is a bit technical and the methods are  somewhat outside of  the scope of this paper.  Regardless, we would like to provide some comments on it. It is well-known, (Theorem XIII.44, p. 204, \cite{RS}) that the Perron-Frobenius property for a self-adjoint operator $\ch$, which is bounded from below, follows from (and it is in many ways equivalent to) the positivity improving property of the associated semigroup $e^{t \ch}$.    It is worth observing that for the standard continuous Schr\"odinger operators $-\Delta+V$, the positivity improving for the semgroup $e^{t (\Delta+V)}$ (and hence of the Perron-Frobenius) is  a direct consequence of the Feynman-Kac formula, which presents the solution to $u_t = \ch v$ as an expectation, against a positive kernel of the initial data. 
In fact, this can be extended to fractional Schr\"odinger operators $\ch=(-\Delta)^s+V, 0<s<1$, via a similar representation formula, in terms of L{\'e}vy processes (instead of Brownian motion). The same representation formulas, while in principle possible, seem unavailable at the moment, hence our  direct proof in the Appendix.

 \subsection{Discrete Szeg\"o inequality} 
 We start with an intuitively clear combinatorial lemma,  which  is somewhat cumbersome to write down. 
 \begin{lemma}
 	\label{le:10} 
 	Let $a_0\geq a_1\geq \ldots \geq a_N\geq 0=a_{N+1}$, be a sequence of non-negative numbers and 
 	$\mu:\{0, \ldots, N+1\}\to \{0, \ldots, N+1\}$ be a permutation. Then, for every $p>1$, 
 	\begin{equation}
 	\label{404} 
 	 \sum_{j=0}^{N} |a_{\mu(j+1)}- a_{\mu(j)}|^p\geq \sum_{j=0}^{N} |a_{j+1}- a_{j}|^p.
 	\end{equation}
 Assuming in addition that the sequence $\{a_j\}_{j=0}^\infty$ is strictly decreasing, we have that equality in \eqref{404} occurs if and only if $\mu=id$. That is, any non-trivial permutation $\mu$ leads to strict inequality in \eqref{404}. 
 \end{lemma}
 \begin{proof}
 Introduce $\ve_j=a_j-a_{j+1}\geq 0, j=0, \ldots, N$, so that $a_j=\ve_j+\ldots +\ve_N$.  We need to show 
 \begin{equation}
 \label{50} 
  \sum_{j=0}^{N} |\ve_{\min(\mu(j+1), \mu(j))} +\ldots+ \ve_{\max(\mu(j+1), \mu(j))-1}|^p\geq \sum_{j=0}^{N} |\ve_j|^p.
 \end{equation}
 Since 
 $$
 |\ve_{\min(\mu(j+1), \mu(j))} +\ldots+ \ve_{\max(\mu(j+1), \mu(j))-1}|^p\geq \ve_{\min(\mu(j+1), \mu(j))}^p +\ldots+ \ve_{\max(\mu(j+1), \mu(j))-1}^p,
 $$
 matters clearly reduce to showing the following statement: 
 for every $k\in \{0, \ldots, N\}$, there exists $j=j_k \in \{0, \ldots, N+1\}$, so that $k\in [\min(\mu(j+1), \mu(j)), \max(\mu(j+1), \mu(j))-1]$. 
 
 We now prove this.
 For $k=0$, choose  $j_0: \mu(j_0)=0$, if $j_0\leq N$ or else, if $\mu(N+1)=0$,  choose $j_0=N$. This  does the job in the case $k=0$. For $k=N$, choose $j_0: \mu(j_0)=N+1$ and this would do it. 
 
Consider now the case $k: 1\leq k\leq N-1$. Introduce  the non-empty disjoint sets 
$A, B: \mu(A)=\{0, \ldots, k\}, \mu(B)=\{k+1, \ldots, N+1\}$. Note that  $\{0, \ldots, N+1\}= A\cup B$. 
If $k\in A$, that is $\mu(k)\leq k$, take $j\in B$, so that $\mu(j)\geq k+1$. Then, it is clear that for some $j_0\in [\min(k,j), \max(k,j)]$, $k\in [\min(\mu(j_0-1), \mu(j_0)), \max(\mu(j_0-1), \mu(j_0))-1]$, so we are done in this case.  Similarly, if $k\in B$, so $\mu(k)\geq k+1$, choose $j\in A$, so that $\mu(j)\leq k$. Again, for some $j_0\in [\min(k,j), \max(k,j)]$, 
$k\in [\min(\mu(j_0-1), \mu(j_0)), \max(\mu(j_0-1), \mu(j_0))-1]$, so we are done again.    

Regarding the equality in \eqref{404}, assuming strictly decreasing sequence (and hence $\ve_j>0, j=0, \ldots, N$), we saw that each $\ve_j$ appears in the sum \eqref{50}. We also saw in the analysis, that if $ \max(\mu(j+1), \mu(j))-\min(\mu(j+1), \mu(j))\geq 2$, for any $j$, then at least one of $\ve_j^p, j=0, \ldots, N$ will appear at least twice leading to strict inequality. Thus, equality in \eqref{404} is possible only if $\max(\mu(j+1), \mu(j))-\min(\mu(j+1), \mu(j))=1$. This is of course possible, only for $\mu=id$. Conversely, there is clearly equality in \eqref{404} for $\mu=id$, so equality in \eqref{404} occurs if and only if $\mu=id$.  
 \end{proof}
 To set up the function spaces, assume $\{f_n\}_{n\in\cz}\in l^p(\cz), 1<p<\infty$. Our next task is to study the minimal possible value of expressions of the form 
 $$
T({\mathbf f})=\sum_{n=-\infty}^\infty |f_{n+1}-f_n|^p, 
 $$
 where we allow ourselves only to permute and translate\footnote{We say that $\{g_n\}$ is a translate of $\{f_n\}$ if there exists $k_0$, so that $g_n=f_{n+k_0}$ for all $n\in\cz$}  the sequence $\{f_n\}$. The reverse triangle inequality $|f_{n+1}-f_n|\geq | |f_{n+1}|-|f_n||$, which is strict if $f_n f_{n+1}<0$ shows that the quantity $T$ is minimized on non-negative sequences. In addition,  it is clear that to minimize $T({\mathbf f})$  effectively, we need to discard redundancies. 
 
 More precisely, we say that   that two sequences $\{f_n\}$ and $\{h_n\}$ are one step equivalent, only if 
 ${\mathbf f} =(\ldots f_{j-1}, f_{j}, f_{j+1},\ldots  )$, $f_{j}=f_{j-1}\neq 0$ and 
 ${\mathbf h}=(\ldots , f_{j-1},  f_{j+1},\ldots)$. That is, ${\mathbf
   f}, {\mathbf h}$ are one step equivalent, if one is obtained from
 the other by erasing one copy of a non-zero  element that repeats
 itself. Note that the erasure operation only allows one to erase a
 repeating element, which is equal to an immediate neighbor.   We say
 that ${\mathbf f}, {\mathbf h}$,   are equivalent, if one is obtained
 from the other by the erasure operation, possibly countably many
 times.  Clearly, for each sequence ${\mathbf f}\in l^p$, there is a
 uniquely determined equivalent element, which has all different values - that is $f_j\neq f_k$ for each $j\neq k$. 
 \begin{proposition}(Discrete Szeg\"o inequality)
 	\label{Szego} 
 	
 	Let $\{f_n\}_{n\in\cz}\in l^p(\cz), 1<p<\infty$. Then, there exists an element $\tilde{{\mathbf f}}\in l^p$, so that 
 	\begin{enumerate}
 		\item $\tilde{{\mathbf f}}$ is a permutation and translation of $\{|f_n|\}$. 
 		\item $\tilde{{\mathbf f}}(0)=\max_{n\in \cz} \tilde{{\mathbf f}}(n)$
 		\item $\tilde{{\mathbf f}}:{\mathbb N}_+\to \rone_+$ is non-increasing, while $\tilde{{\mathbf f}}:{\mathbb N}_-\to \rone_+$ is non-decreasing. 
 	\end{enumerate}
 	for which, we have 
 	\begin{equation}
 	\label{602} 
 		\sum_{n=-\infty}^\infty |f_{n+1}-f_n|^p\geq \sum_{n=-\infty}^\infty |\tilde{f}_{n+1}-\tilde{f}_n|^p,
 	\end{equation}
 	Equality in \eqref{602} holds if and only if $f_n$ does not change sign and for some $j_0\in\cz$,
 	$$
 |f_{j_0}|\geq |f_{j_0-1}|\geq \ldots  	|f_{j_0-n}|\ldots;   \ \  |f_{j_0}|\geq |f_{j_0+1}|\geq |f_{j_0+2} |\geq \ldots |f_{j_0+n}|\geq \ldots 
 	$$
 \end{proposition}
 \begin{proof}
 	As we have argued above, 
 	$$
 	\sum_{n=-\infty}^\infty |f_{n+1}-f_n|^p\geq \sum_{n=-\infty}^\infty ||f_{n+1}|-|f_n||^p,
 	$$
 	with strict inequality, if there is a change of sign $f_{n+1} f_n<0$. So, we might reduce to the case $f_n\geq 0$. Also, since ${\mathbf f}\in l^p$, there is a maximal element, which we assume, by translational invariance to be at $n=0$. That is, $f_n(0)=\max_{n\in \cz} f(n)$. Next, it suffices to consider such positive sequences with compact support, say on $[0, N]$, and $f_n(0)=\max_{n\in \cz} f(n)$, for which we need to show 
 	\begin{equation}
 	\label{70} 
 		\sum_{n=0}^N |f_{n}-f_{n+1}|^p\geq \sum_{n=0}^N |f^*_{n}-f^*_{n+1}|^p. 
 	\end{equation}
 	where $\{f_n^*\}$ is the decreasing rearrangement of $f_n, 0\leq n\leq N$. For the proof of \eqref{70}, we simply refer to Lemma \ref{le:10}. Indeed, we start with $a_n=f_n^*, n=0, \ldots, N, a_{N+1}=0$ and a permutation $\mu: \{0,\ldots, N+1\}\to \{0,\ldots, N+1\}$, defined so that 
 	$f_n=f^*_{\mu(n)}$. 
 	
 	 Based on \eqref{70}, we can also prove 
 	\begin{equation}
 	\label{75} 
 		\sum_{n=-N}^{-1} |f_{n}-f_{n+1}|^p\geq \sum_{n=-N}^{-1}  |f^*_{n}-f^*_{n+1}|^p.
 	\end{equation}
 	After combining the two and taking the  limit $N\to \infty$, we obtain \eqref{602}, where $\tilde{{\mathbf f}}$ is the decreasing rearrangement of ${\mathbf f}$. 
 	
 	The equality occurs, if it occurs for both \eqref{70} and \eqref{75}. Applying the criteria for equality in Lemma \ref{le:10} implies that one must have, say for the case $f_n\geq 0$ and $f_0=\max f(n)$, that $f_0>f_1\ldots f_n \ldots$ and $f_0>f_{-1}\ldots $. Clearly, in such a case, equality in \eqref{602} holds true. 

 \end{proof}

 \section{Construction of the normalized waves}
 We start with some basic properties of the variational problem \eqref{30}.
 \subsection{Properties of the variational problem \eqref{30}}
 The following  estimate    is instrumental to  the well-posedness of the variational problem \eqref{30}. In fact, this is a version of a  discrete Sobolev embedding. 
 \begin{lemma}
 	\label{le:43} 
 	Let $d\geq 1$ and $2<p\leq p_d=\left\{\begin{array}{cc} 
 	+\infty & d=1,2 \\
 	\f{2d}{d-2} & d\geq 3
 	\end{array}  \right.$. 
 
 	Then,  for every $0<\theta<\min(1, d (\f{1}{2}-\f{1}{p}))$, there exists $C_{\theta, p}$,  so that for every sequence $\ubb\in l^2(\cz^d)$
 	\begin{equation}
 	\label{60} 
 	\left(\sum_{n\in \cz^d} |u_n|^p\right)^{1/p}  \leq C_\ve \left(\sum_{j\in \cz^d: |j-n|=1} |u_j-u_n|^2\right)^{\f{\theta}{2}} 
 	\left(\sum_{n\in \cz^d} |u_n|^2\right)^{\f{1}{2}-\f{\theta}{2}}
 	\end{equation}
 	For $p>p_d$, \eqref{60}holds with $\theta=1$. 
 \end{lemma}
 {\bf Remark:} Note that since $\sum_{j\in \cz^d: |j-n|=1} |u_j-u_n|^2\leq C_d \sum_n |u_n|^2$, the estimates with higher $\theta$ imply the ones with smaller $\theta$. 
 \begin{proof}
 	Per our earlier computations, see Section \ref{sec:2.1}, we have that 
 	$$
 	\sum_{j\in \cz^d: |j-n|=1} |u_j-u_n|^2=\dpr{-\De_{disc} {\mathbf u}}{{\mathbf u}}=4 \int_{[0,1]^d} |{\mathbf u}(\xi)|^2 
 	\left(\sum_{m=1}^d \sin^2(\pi \xi_m) \right)d\xi
 	$$
 	Denoting the symbol $\De(\xi):=\sum_{m=1}^d \sin^2(\pi \xi_m)$, we see that \eqref{60} is better expressed in the following form 
 	\begin{equation}
 	\label{65} 
 	\left(\sum_{n\in \cz^d} |u_n|^p\right)^{1/p}  \leq C_\ve \left(\int_{[0,1]^d} |{\mathbf u}(\xi)|^2 
 	\De(\xi) d\xi\right)^{\f{\theta}{2}} 
 	\left(\int_{[0,1]^d} |{\mathbf u}(\xi)|^2 d\xi
 	 \right)^{\f{1}{2}-\f{\theta}{2}},
 	\end{equation}
 	which we now prove. Starting with the Hausdorff-Young's inequality, we obtain 
 	$$
 	\left(\sum_{n\in \cz^d} |u_n|^p\right)^{1/p}\leq C \|{\mathbf u}\|_{L^{p'}([0,1]^d)}.
 	$$
 	Next, observe that since 
 	$$
 	\De(\xi)\geq c\min(|\xi|^2,
        \sum_{m=1}^d |\xi_m-1|^2),
        $$
        we have that 
 	$$
 	\int_{[0,1]^d} \f{1}{\De(\xi)^\ga} d\xi<C_\ga, 
 	$$
  for all $\ga<\f{d}{2}$. Thus, we can estimate by H\"older's as follows 
  $$
  \|{\mathbf u}\|_{L^{p'}([0,1]^d)}\leq  \left(\int_{[0,1]^d} |{\mathbf u}(\xi)|^2 \De(\xi)^{\f{2\al}{p'}}\right)^{\f{1}{2}} \left(\int_{[0,1]^d} \f{1}{\De(\xi)^{ \al q} } d\xi\right),
  $$
 	where $q: \f{1}{q}+\f{p'}{2}=1, \al>0$. Note that we select $\al: \al q<\f{d}{2}$, so that the second integral is finite. If we can select $\al: \f{2\al}{p'}\geq 1$ (noting that $\De(\xi)<d$), then we have shown \eqref{60} with $\theta=1$. This happens in the case $p>p_d$. 
 	
 	In the case $p: 2<p\leq p_d$, the inequality $\al<\f{d}{2q}$ implies $ \f{2\al}{p'}<1$. Then,  again by H\"older's, 
 $$
 		\left(\int_{[0,1]^d} |{\mathbf u}(\xi)|^2 \De(\xi)^{\f{2\al}{p'}}\right)^{\f{1}{2}}\leq \left(\int_{[0,1]^d} |{\mathbf u}(\xi)|^2 \De(\xi) d\xi \right)^{\f{\al}{p'}} \left(\int_{[0,1]^d} |{\mathbf u}(\xi)|^2 d\xi  \right)^{\f{1}{2}-\f{\al}{p'}}. 
 $$
 Putting all the inequalities together yields for $\theta=\f{2\al}{p'}<1$, 
 	\begin{equation}
 	\label{703} 
 \left(\sum_{n\in \cz^d} |u_n|^p\right)^{1/p}\leq C_\theta \left(\int_{[0,1]^d} |{\mathbf u}(\xi)|^2 \De(\xi) d\xi \right)^{\f{\theta}{2}} \left(\int_{[0,1]^d} |{\mathbf u}(\xi)|^2 d\xi  \right)^{\f{1}{2}-\f{\theta}{2}}. 
\end{equation}
 	As $\al$ can assume all values $\al<\f{d}{2q}=\f{d}{2}(1-\f{p'}{2})$, then $\theta$ can be taken in \eqref{703} to be $\theta<d (\f{1}{2}-\f{1}{p})$.  
 \end{proof}
 Based on Lemma \ref{le:43}, it can be shown that  the variational problem \eqref{30} has solutions, for all values of $\la>0$, if $\si<\f{2}{d}$. On the other hand, it can be shown that for $\si\geq \f{2}{d}$, solutions exist only for large enough values of $\la>0$. This has been shown by Weinstein, see Propositions 4.1 and 4.2, \cite{Wein}. His proof is based on the availability of the estimate \eqref{60}, in the case $p=2\si+2$, with $\theta: (2\si+2)\f{1-\theta}{2}<\sigma$, see the question (4.1) in \cite{Wein}.  This amounts to exactly $\si<\f{2}{d}$. 
 
 In addition, Weinstein, \cite{Wein} has shown that for $\si>\f{2}{d}$, the problem \eqref{30} is not well-posed, in the sense that for small enough $\la>0$, $\inf\limits_{\sum_n |u_n|^2 =\la} H[\ubb]=-\infty$. 
 
 Our next lemma concerns $h(\la)$. 
 \begin{lemma}
 	\label{le:45} 
 	Let $d\geq 1$, $\si<\f{2}{d}$ and $\la>0$. Then, $h(\la):=\inf\limits_{\sum_n |u_n|^2 =\la} H[\ubb]<0$. 
 \end{lemma}
 {\bf Remark:} This has already been shown by Weinstein, \cite{Wein},  but we offer alternative proof. 
 \begin{proof}
 	We need to produce an example, for which $\sum_n |u_n|^2 =\la$, but $H[\ubb]<0$. For large $N>>1$, consider 
 	$$
 	u_{k_1, \ldots, k_d}:=\left\{ \begin{array}{cc} 
 	 c_{d, \la,N} \left[\f{1}{N^{\f{d}{2}}}  - \f{|k_1|+\ldots + |k_d|}{N^{1+\f{d}{2}}}\right] &  |k_1|+|k_2|+\ldots+|k_d|\leq N-1 \\
 	 0 & |k_1|+|k_2|+\ldots+|k_d|\geq N.
  \end{array}
 	\right.
 	$$
 	where $c_{d,\la,N}$ so that $\|\ubb\|_{L^2}^2=\la$. Note that for 
 	$|u_{k_1, \ldots, k_d}|\leq const.  N^{-d/2}$, while $|u_{k_1,
          \ldots, k_d}|\sim N^{-d/2}$ for $|k_1|+\ldots |k_d|<N/2$,
        we have that 
 	$$
 	\sum_{k_1, \ldots, k_d} |u_{k_1, \ldots, k_d}|^2\sim c_{d, \la,N}^2, 
 	$$
 	whence $c_{d,\la, N}\sim c_{\la, d}$. That is, it is bounded above and below by constants independent of $N$. 
 	
 	Next, for $|k_1|+|k_2|+\ldots+|k_d|\leq N-2$, we have that for all $j: |j-k|=1$, 
 	$|u_{k}-u_{j}|=c_{d,\la,N} N^{-1-d/2}$, while for the boundary $|k_1|+|k_2|+\ldots+|k_d|=N-1$ and $|k_1|+|k_2|+\ldots+|k_d|=N$, we also have that $|u_{k}-u_{j}|\leq c_{d,\la,N} N^{-1-d/2}$ for all $j: |j-k|=1$. Thus, 
 	$$
 	\sum_{j\in \cz^d: |j-n|=1} |u_j-u_n|^2 \leq c_{d,\la,N} N^d N^{-2-d}\leq c_{d,\la} N^{-2}
 	$$
 	Also, it is clear 
 	$$
 	\sum_{n\in\cz^d} |u_n|^{2\si+2} \sim N^d N^{-\f{d}{2}(2\si+2)}\sim N^{-d\si}.
 	$$
 	All in all, we conclude 
 	$$
 	H[\ubb]\leq C_{\la,d} N^{-2} - \tilde{C}_{\la, d} N^{-d\si}
 	$$
 	Clearly, if $\si<\f{2}{d}$, $H[\ubb]$ can be made negative, if $N=N_{\la,d}$ is selected large enough. 
 \end{proof}
 
 \begin{figure}[t]
  \includegraphics[scale=0.45]{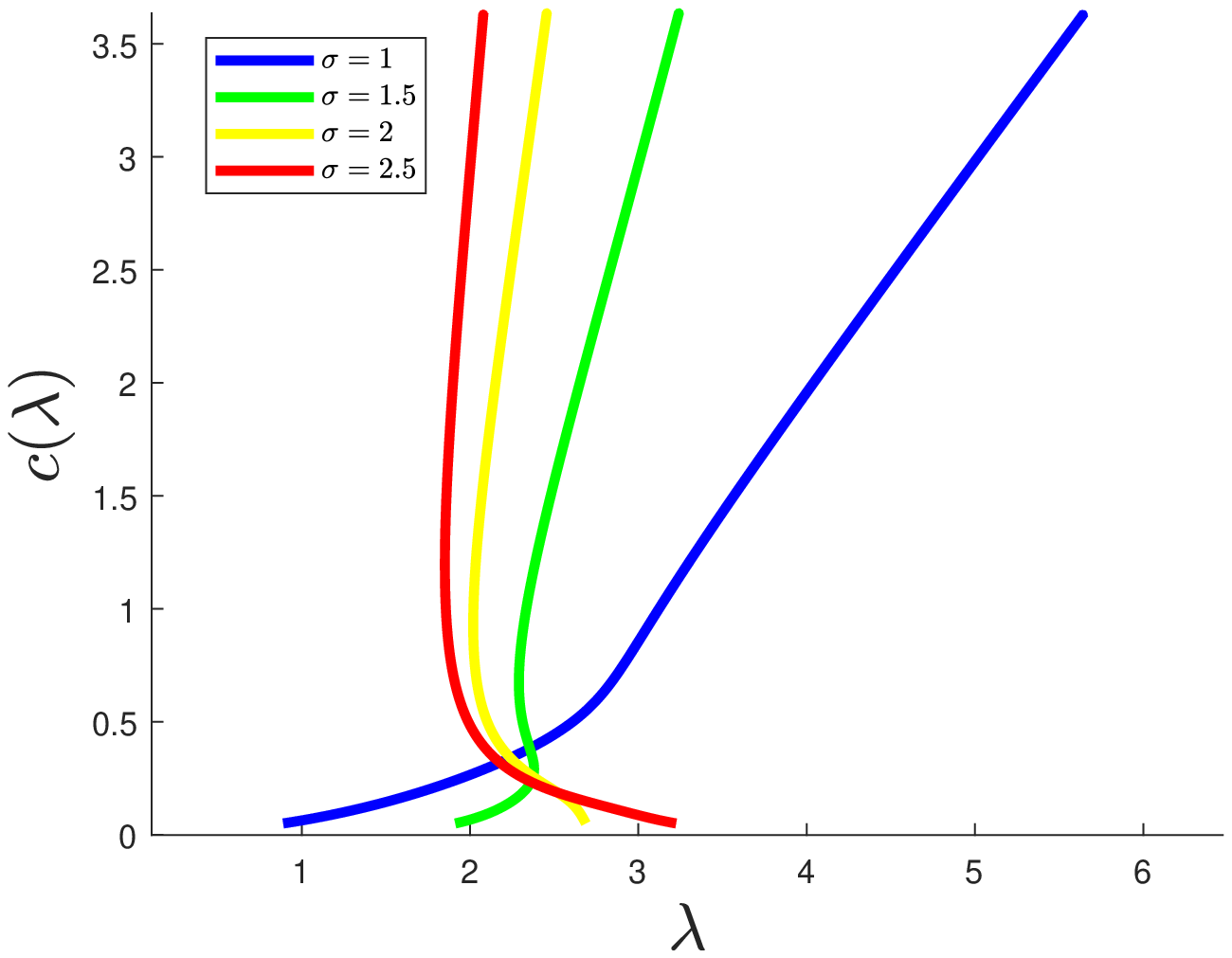}
   \includegraphics[scale=0.45]{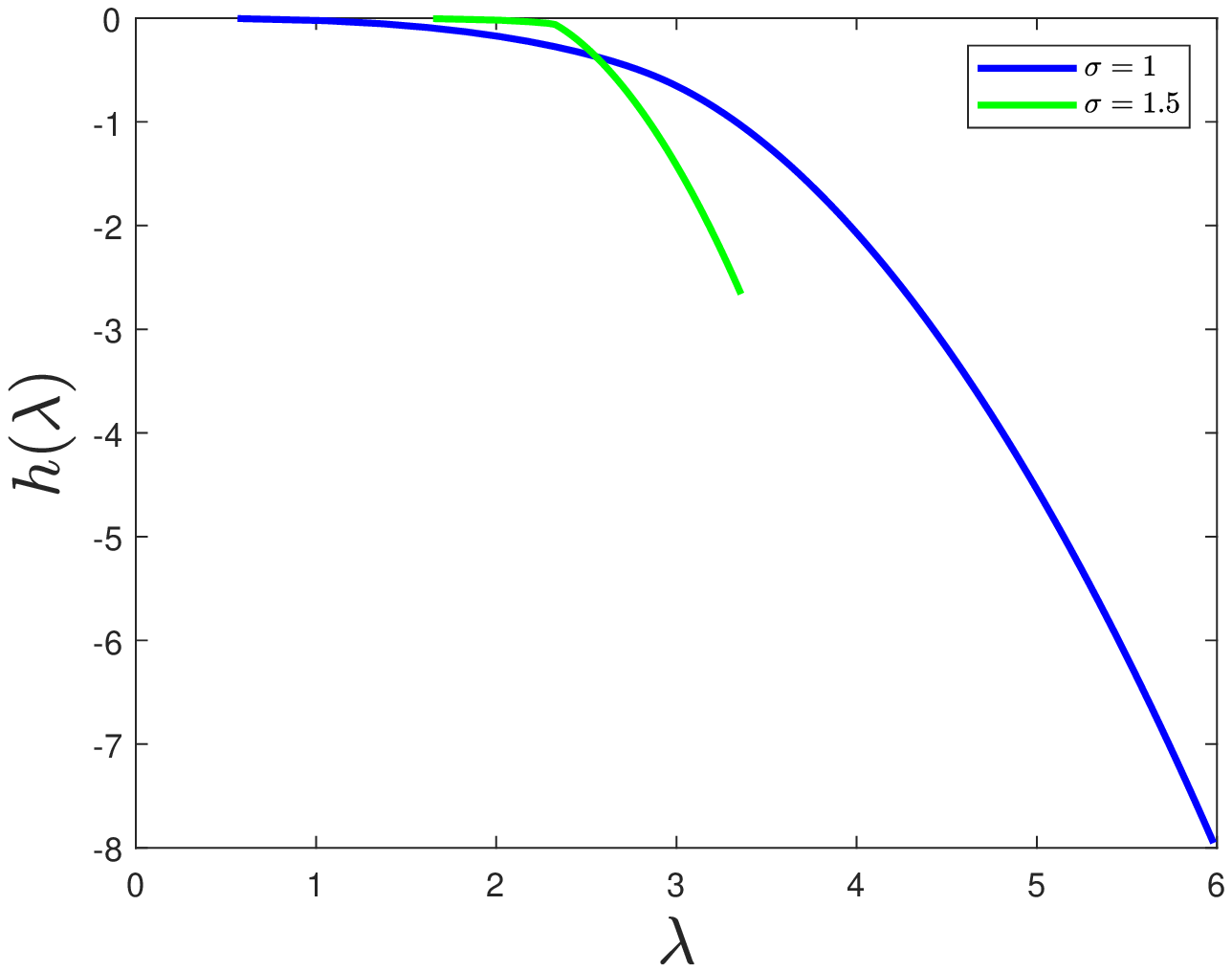}
  \caption{The functions $c(\la)$ and $h(\la)$ in dimension $d=1$.}
  \label{fig-chlm}
  \end{figure}
  
 The next lemma shows some further useful  properties of $h$, in
 particular the sublinearity of $h$, the concavity of $h$ and the
 Lipschitz properties of $h$.
 
 \begin{lemma}
 	\label{le:subad} 
 	The function $h:\rone_+\to \rone$ is locally Lipschitz continuous and concave. 
 	
 	In addition, it is sub-linear. That is for 
  $\la>0$ and $\al\in (0, \la)$, we have 
 	$$
 	h(\la)<h(\al)+h(\la-\al).
 	$$
 \end{lemma}
 \begin{proof}
 	For the Lipschitz property, we write 
 	$$
 	h(\la) = \la \inf\limits_{\sum_n |u_n|^2 =1} \left[\sum_{j\in \cz^d: |j-n|=1} |u_j-u_n|^2 - \f{\la^\si}{\si+1} \sum_{n\in\cz^d} |u_n|^{2\si+2}\right]=: \la \inf\limits_{\|\ubb\| =1} \tilde{H}_\la [\ubb].
 	$$
 	It clearly suffices to show that $\tilde{h}:=\inf\limits_{\|\ubb\| =1} 
 	\tilde{H}_\la [\ubb]$ is Lipschitz. We have 
 	$$
 	|	\tilde{H}_{\la+\de} [\ubb] - 	\tilde{H}_\la [\ubb]|\leq \f{1}{\si+1} ((\la+\de)^{\si}-\la^{\si}) \sum_{n\in\cz^d} |u_n|^{2\si+2}\leq C_\la \de, 
 	$$
 	for all $\ubb: \|\ubb\|_{l^2}=1$. It follows that for all $\ubb: \|\ubb\|_{l^2}=1$, there is 
 	$$
 		\tilde{H}_\la [\ubb]-C_\la \de\leq 	\tilde{H}_{\la+\de} [\ubb]\leq 	\tilde{H}_\la [\ubb]+C_\la \de
 	$$
 Taking $\inf_{\|\ubb\|_{l^2}=1}$ implies
 $|\tilde{h}(\la+\de)-\tilde{h}(\la)|\leq C_\la \de$, whence the
 Lipschitz property of $h$ follows. 
 
 For the concavity, recall the formula 
 \begin{eqnarray*}
 h(\la)=\inf\limits_{\sum_n |u_n|^2 =1}\left[ \la \sum_{j\in \cz^d: |j-n|=1} |u_j-u_n|^2 - \f{\la^{\si+1}}{\si+1} \sum_{n\in\cz^d} |u_n|^{2\si+2}\right].
 \end{eqnarray*}
 	For $a\in (0,1)$, $\la_1, \la_2>0$, we have due to the convexity of the function $\la\to \la^{\si+1}$, 
 	 \begin{eqnarray*}
 & & 	(a\la_1+(1-a)\la_2)  \sum_{j\in \cz^d: |j-n|=1} |u_j-u_n|^2 - 
 	\f{(a\la_1+(1-a) \la_2)^{\si+1}}{\si+1} \sum_{n\in\cz^d} |u_n|^{2\si+2}\geq \\
 	&\geq & a \left[\la_1 \sum_{j\in \cz^d: |j-n|=1} |u_j-u_n|^2-
 	\f{\la_1^{\si+1}}{\si+1} \sum_{n\in\cz^d} |u_n|^{2\si+2}\right]+\\
 	&+& (1-a) \left[\la_2 \sum_{j\in \cz^d: |j-n|=1} |u_j-u_n|^2-
 	\f{\la_2^{\si+1}}{\si+1} \sum_{n\in\cz^d} |u_n|^{2\si+2}\right].
 	 \end{eqnarray*}
 	Taking infimum on $\ubb: \|\ubb\|_{l^2}=1$, we obtain the inequality 
 	$$
 	h(a\la_1+(1-a)\la_2)\geq a h(\la_1)+(1-a) h(\la_2),
 	$$
 	 which is the desired concavity. 
 	
 Finally, regarding the sub-additivity, observe that 
 	\begin{eqnarray*} 
 	h(\la) &=& \inf\limits_{\sum_n |u_n|^2 =\la} \left[\sum_{j\in \cz^d: |j-n|=1} |u_j-u_n|^2 - \f{1}{\si+1} \sum_{n\in\cz^d} |u_n|^{2\si+2} \right]= \\
 	&=& \f{\la}{\al} \inf\limits_{\sum_n |u_n|^2 =\al} \left[\sum_{j\in \cz^d: |j-n|=1} |u_j-u_n|^2 - \f{(\la/\al)^\si}{\si+1} \sum_{n\in\cz^d} |u_n|^{2\si+2}\right]<
 	\f{\la}{\al} h(\al),
 	\end{eqnarray*}
 	where the last inequality is strict, because $\al<\la$ and there is a minimizing sequence $\ubb^m$ for \eqref{30}, with the property 
 	$\limsup_m\|\ubb^m\|_{l^{2\si+2}}>0$. Indeed, otherwise, if there is a minimizing sequence $\ubb^m: \lim_m \|\ubb^m\|_{l^{2\si+2}}=0$, then, we would have had 
 	$$
 	h(\la)=\lim_m   \sum_{j\in \cz^d: |j-n|=1} |u^m_j-u^m_n|^2 \geq 0,
 	$$
 	 in contradiction  with Lemma \ref{le:45}. 
 	 
 	 Now that we have established $h(\la)<	\f{\la}{\al} h(\al)$, note that this means that $\la\to \f{h(\la)}{\la}$ is strictly decreasing\footnote{This conclusion can also be easily reached for differentiable functions $h$, which are concave, as was established earlier. Unfortunately, the differentiability of $h$ is only valid a.e., and we prefer to give the direct argument instead.}. Assume that $\al\in [\f{\la}{2}, \la)$ (and otherwise, work with $\la-\al$ instead of $\al$). We get 
 	 $$
 	 h(\la)<	\f{\la}{\al} h(\al)=h(\al)+ \f{\la-\al}{\al} h(\al)\leq h(\al)+h(\la-\al),
 	 $$
 	 where in the last inequality, we have used that since $\la-\al\leq \f{\la}{2}\leq \al$, one has $\f{h(\al)}{\al} \leq \f{h(\la-\al)}{\la-\al}$. 
 \end{proof}
 We are now ready for the existence result. 
 \subsection{Conclusion of the proof of Theorem \ref{theo:5}}
 We show that 
 	if  $d\geq 1$, $0<\si<\f{2}{d}$ and $\la>0$, then the variational problem \eqref{30} has a solution as described in Theorem \ref{theo:5}. 
 
 	The proof consists of applying the method of compensated compactness. Indeed, take a minimizing sequence $\{{\mathbf u^k}\}$ for \eqref{30}. That is, $\|{\mathbf u^k}\|_{l^2}^2=\la$ and $H[{\mathbf u^k}]\to h(\la)$. For every sequence with the property, $\|{\mathbf u^k}\|_{l^2}^2=\la$, standard arguments show that one of three alternatives takes place: 
 \begin{enumerate}
 	\item (Tightness) There exists a sequence $n_k\in \cz^d$, so that for any $\epsilon>0$, there exists $R=R(\epsilon)$, so that 
 	$$
 	\sum_{n\in \cz^d: |n-n_k|<R_\epsilon} |u_n^k|^2>\la-\epsilon. 
 	$$
 	\item (Vanishing) For every $R>0$,
 	$$
 	\lim_{k\to \infty} \sup_{y\in \cz^d} \sum_{n\in \cz^d: |n-y|<R} |u_n^k|^2=0.
 	$$
 	\item (Splitting)  There exists $\al\in (0, \la)$, so that for any $\epsilon>0$, there is $R$ and $R_k\to \infty$ and $n_k$, so that 
 	$$
 	\left|\sum_{n\in \cz^d: |n-n_k|<R} |u_n^k|^2-\al\right|\leq \epsilon, 
 	\left|\sum_{n\in \cz^d: |n-n_k|>R_k} |u_n^k|^2-(\la-\al)\right|\leq \epsilon
 	$$
 \end{enumerate}
 	We will show that the minimizing sequence introduced above, cannot be anything but tight. Indeed, assume vanishing. In particular, this means that $\lim_k \|{\mathbf u^k}\|_{l^\infty}=0$. Then, 
 	$$
 	\sum_{n} |u_n^k|^{2\si+2}\leq \|{\mathbf u^k}\|_{l^\infty}^{2\si} \|{\mathbf u^k}\|_{l^2}^2,
 	$$
 	whence $\lim_k \|{\mathbf u^k}\|_{l^{2\si+2}}=0$. Then, 
 	$$
 	h(\la)=\lim_k H[{\mathbf u^k}]=\lim_k \sum_{j\in \cz^d: |j-n|=1} |u^k_j-u^k_n|^2\geq 0,
 	$$
 	a contradiction with Lemma \ref{le:45}. 
 	
 	Next, assume splitting. Denote 
 	$$
 	{\mathbf v^k}={\mathbf u^k}\chi_{|n-n_k|<R}, {\mathbf w^k}={\mathbf u^k}\chi_{|n-n_k|>R_k}, {\mathbf z^k}={\mathbf u^k}\chi_{R<|n-n_k|<R_k}
 	$$
 	Clearly, we have that $\|{\mathbf v^k}\|_{l^2}^2\in (\al-\epsilon, \al+\epsilon), 
 	\|{\mathbf w^k}\|_{l^2}^2 \in (\la-\al-\epsilon, \la-\al+\epsilon)$, while $\|{\mathbf z^k}\|_{l^2}^2\leq \epsilon$. Also, by the definition of the functional $H$, there exists an absolute  constant $C>0$, so that 
 	$$
 	H[{\mathbf u^k}]=H[{\mathbf v^k}]+H[{\mathbf w^k}]+O(\epsilon)\geq h(\al+\epsilon)+h(\la-\al+\epsilon)-C\epsilon. 
 	$$
 	Taking limits in $k\to \infty$ and using the fact that $\la\to h(\la)$ is decreasing, 
 	we obtain the inequality 
 	$$
 	h(\la)\geq h(\al+\epsilon)+h(\la-\al+\epsilon)-C\epsilon,
 	$$
 	which is true for arbitrary $\epsilon>0$. By the continuity of $h$, it follows that 
 	$h(\la)\geq h(\al)+h(\la-\al)$, which is in contradiction with Lemma \ref{le:subad}. Thus, it remains that ${\mathbf u^k}$ is tight, whence, it has a convergent subsequence, which is clearly a solution to \eqref{30}. 
 	 The inequalities \eqref{ineq1}, \eqref{ineq2} follow by  an application of the discrete Szeg\"o inequality Proposition \ref{Szego} for each fixed \\  $(k_1, \ldots, k_{j_0-1}, k_{j_0+1}, \ldots, k_d)$. 
 	
 	The Euler-Lagrange equation \eqref{EL:10} is derived in the usual fashion. Namely, starting with a constrained minimizer $\ubb^0=(u_n)_{n\in\cz^d}:\|\ubb^0\|_{l^2}^2=\la$, we consider a perturbation $\ubb^0+\ve \ubq, \ubq\in l^2$ and then the scalar function 
 	\begin{eqnarray*}
 	f(\ve) &:=& H\left[\sqrt{\la}\f{\ubb^0+\ve \ubq}{\|\ubb^0+\ve \ubq\|} \right] = \\
 	&=&  \f{\la}{\|\ubb^0+\ve \ubq\|^2}  \dpr{-\De_{disc}(\ubb^0+\ve \ubq) }{\ubb^0+\ve \ubq} - \f{1}{\si+1} \left(\f{\la}{\|\ubb^0+\ve \ubq\|^2} \right)^{\si+1} \sum_{n\in\cz^d} |u_n+\ve q_n|^{2\si+2}.
 	\end{eqnarray*}
 	
 	Clearly $f(0)=h(\la)=H[\ubb^0]=\inf_{\|\ubb\|^2=\la} H[\ubb]$, whence $ 0$ is a local minimum for the function $f$ and so $f'(0)=0$. Writing the expansion in powers of $\ve$ and isolating the coefficient in front of $\ve$ (which must be zero), brings about the equation 
 	\begin{equation}
 	\label{105} 
 		\sum_n (-\De_{disc} \ubb_n+c(\la) u_n -  u_n^{2\si+1})q_n=0, 
 	\end{equation} 
 	where  $\la c(\la) =\sum_{n\in \cz^d} |u_n|^{2\si+2} -  \sum_{j\in \cz^d: |j-n|=1} |u_j-u_n|^2$.  Noting that \eqref{105} must be valid for arbitrary $q$, we conclude that \eqref{EL:10} holds. Regarding the sign of $c(\la)$, note that since $h(\la)=\sum_{j\in \cz^d: |j-n|=1} |u_j-u_n|^2 - \f{1}{\si+1} \sum_{n\in \cz^d} |u_n|^{2\si+2}<0$, it follows that 
 	$$
 	\la c(\la)=\f{\si}{\si+1} \sum_{n\in \cz^d} |u_n|^{2\si+2}-h(\la)>0.
 	$$
 	For the second variation claim, consider again the function $f$, with $\ubq\perp \ubb^0, \|\ubq\|_{l^2}=1$. That is $\dpr{\ubb^0}{\ubq}=0$. In such case,  $\|\ubb^0+\ve \ubq\|^2=\|\ubb^0\|^2+\ve^2=\la+\ve^2$. We take advantage of the fact that $0$ is a minimum for $f$.  It now implies $f''(0)\geq 0$. The coefficient in front of $\ve^2$ in the Taylor expansion of $f(\ve)$ around zero is 
 	$$
 	\dpr{-\De_{disc} \ubq}{\ubq}+c(\la) \dpr{\ubq}{\ubq} - (2\si+1) \sum_n u_n^{2\si} |q_n|^2.
 	$$
 	 Hence, we conclude that for every $\ubq: \|\ubq\|=1, \ubq\perp \ubb^0$, one has 
 	 $$
 	 \dpr{\cl_+\ubq}{\ubq}=	\dpr{-\De_{disc} \ubq}{\ubq}+c(\la) \dpr{\ubq}{\ubq} - (2\si+1) \sum_n u_n^{2\si} |q_n|^2>0
 	 $$
 	 This means that $\cl_+|_{\{\ubb^0\}^\perp}\geq 0$. 
 	 
 	 \subsection{Proof of Theorem  \ref{cor:17}}
 	 We have already established that ( see Theorem \ref{theo:5}), that $n(\cl_+)=1$, and in fact $\cl_+|_{\{\ubb\}^\perp}\geq 0$. The next task to is to discuss the spectral picture for the other linearized operator $\cl_-$. We show that $\cl_-\geq 0$, with a simple eigenvalue at zero, with $Ker(\cl_-)=span[\ubb]$. 
 	 
 	 To this end,  and by a direct verification shows that $\cl_-[\ubb]=0$, this is simply the profile equation \eqref{EL:10}. By construction, $\ubb\geq 0$. Note that $\cl_-c(\la)$ satisfies the requirements of 
 	  Theorem \ref{prop:PF}, so we know that the lowest eigenvalue
          of $\cl_-$ is a simple eigenvalue, with
          an  eigenfunction with positive entries. We claim that this eigenvalue is zero. Indeed, zero is an eigenvalue, and assume that it is not the smallest one. Then, there is a negative eigenvalue for $\cl_-$. But then, according to Theorem \ref{prop:PF}, this eigenvalue will have an eigenfunction (ground state), with non-negative entries. But this causes a contradiction as the ground state needs to be orthogonal to $\ubb$, as eigenfunctions corresponding to different eigenvalues. This shows that zero is at the bottom of the spectrum for $\cl_-$, so $\cl_-\geq 0$, and there is $\de>0$, so that $\cl_-|_{\{\ubb\}^\perp}\geq \de$. 

  Now that we know that $n(\cl_+)=1, n(\cl_-)=0$,
  the general instability index theory, as developed in \cite{KKS1, KKS2, Kap, Pel},  stability will be established as follows.  First, we shall    verify that 
 	  $\ubb \perp Ker(\cl_+)$ and then, as the element $\cl_+^{-1} \ubb$ is defined uniquely in $Ker(\cl_+)^\perp$, 
 	  the following quantity  $\dpr{\cl_+^{-1} \ubb}{\ubb}$ is uniquely defined as well. According to the instability index theory, \cite{KKS1, KKS2, Kap, Pel}, spectral stability is a consequence of $\dpr{\cl_+^{-1} \ubb}{\ubb}<0$. 
 	  
 	  To this end, we show that $\ubb\perp Ker(\cl_+)$. Take an element $\psi\in Ker(\cl_+), \|\psi\|=1$. Let 
 	  $$
 	  \eta:=\psi- \|\ubb\|^{-2} \dpr{\psi}{\ubb} \ubb\perp \ubb.
 	  $$
 	 As $\cl_+|_{\{\ubb\}^\perp}\geq 0$. it must be that 
 	 \begin{equation}
 	 	\label{par:32} 
 	 	0\leq \dpr{\cl_+ \eta}{\eta} = \|\ubb\|^{-4}  \dpr{\psi}{\ubb}^2 \dpr{\cl_+ \ubb}{\ubb}.
 	 \end{equation}
 	 Recall however that $\dpr{\cl_+ \ubb}{\ubb}=-2\si \sum_{n} |u_n|^{2\si+2}<0$, whence \eqref{par:32} is contradictory, unless $\dpr{\psi}{\ubb}=0$. So, we have established that $\ubb\perp Ker(\cl_+)$. As we have argued already, it follows that  $\cl_+^{-1} \ubb$ is defined uniquely in $Ker(\cl_+)^\perp$. Introduce 
 	 $$
 	 {\mathbf h}:=\cl_+^{-1} \ubb - \|\ubb\|^{-2} \dpr{ \cl_+^{-1} \ubb}{\ubb} \ubb\perp \ubb 
 	 $$

 	 Again, it must be that $\dpr{\cl_+ {\mathbf h}}{{\mathbf h}}\geq 0$. But, by simple computation, 
 	 $$
 	 0\leq \dpr{\cl_+ {\mathbf h}}{{\mathbf h}} = -\dpr{\cl_+^{-1} \ubb}{\ubb}+ \|\ubb\|^{-4} \dpr{ \cl_+^{-1} \ubb}{\ubb}^2 \dpr{\cl_+ \ubb}{\ubb}\leq -\dpr{\cl_+^{-1} \ubb}{\ubb}, 
 	 $$
 	 since as established already $\dpr{\cl_+ \ubb}{\ubb}<0$. It follows that $\dpr{\cl_+^{-1} \ubb}{\ubb}\leq 0$. 
 	 Under the extra assumption $\dpr{\cl_+^{-1} \ubb}{\ubb}\neq 0$, it follows that $\dpr{\cl_+^{-1} \ubb}{\ubb}<0$, and the spectral stability follows. 

 	 \section{Existence and stability properties of the minimizers of homogeneous functionals}
 We first provide the proof of Theorem \ref{theo:10}.
 \subsection{Proof of Theorem \ref{theo:10}}
 We apply the compensation compactness arguments again, but this time
 for  sequences in the $l^{2\si+2}$ norms. As before, pick a
 minimizing sequence for \eqref{40}, $\|\ubb^n\|_{l^{2\si+2}}=1$,
 $J[\ubb^n]\to j(\om)$. We have three alternatives

  \begin{enumerate}
  	\item (Tightness) There exists a sequence $n_k\in \cz^d$, so that for any $\epsilon>0$, there exists $R=R(\epsilon)$, so that 
  	$$
  	\sum_{n\in \cz^d: |n-n_k|<R_\epsilon} |u_n^k|^{2\si+2}>1-\epsilon. 
  	$$
  	\item (Vanishing) For every $R>0$,
  	$$
  	\lim_{k\to \infty} \sup_{y\in \cz^d} \sum_{n\in \cz^d: |n-y|<R} |u_n^k|^{2\si+2}=0.
  	$$
  	\item (Splitting)  There exists $\al\in (0, \la)$, so that for any $\epsilon>0$, there is $R$ and $R_k\to \infty$ and $n_k$, so that 
  	$$
  	\left|\sum_{n\in \cz^d: |n-n_k|<R} |u_n^k|^{2\si+2}-\al\right|\leq \epsilon, 
  	\left|\sum_{n\in \cz^d: |n-n_k|>R_k} |u_n^k|^{2\si+2}-(1-\al)\right|\leq \epsilon
  	$$
  \end{enumerate}
 The vanishing is handled similarly as before. More concretely, the vanishing implies \\ $\lim_k \|\ubb^k\|_{l^\infty}\to 0$. But then,  since $J[\ubb^k]\geq \om \|\ubb^k\|_{l^2}^2$, we have that $\sup_k \|\ubb^k\|_{l^2}<\infty$, whence 
 $$
 1=\|\ubb^k\|_{l^{2\si+2}}^{2\si+2}\leq \|\ubb^k\|_{l^\infty}^{2\si} \|\ubb^k\|_{l^2}^2
 $$
  is contradictory for large $k$. 
  
  For splitting, we  take again 
  $$
  {\mathbf v^k}={\mathbf u^k}\chi_{|n-n_k|<R}, {\mathbf w^k}={\mathbf u^k}\chi_{|n-n_k|>R_k}, {\mathbf z^k}={\mathbf u^k}\chi_{R<|n-n_k|<R_k}
  $$
  so that 
  $$
  |\|{\mathbf v^k}\|^{2\si+2}-\al|\leq \epsilon, |\|{\mathbf w^k}\|^{2\si+2}-(1-\al)|\leq \epsilon, |\|{\mathbf z^k}\|^{2\si+2}-(1-\al)|\leq \epsilon
  $$
 and we have the estimate 
 $
 J[\ubb^k]\geq J[{\mathbf v^k}]+ J[{\mathbf w^k}]-C \epsilon.
 $
 Observe that a simple rescaling argument shows that for all $\al>0$, we have 
 $$
 j_\al (\om):=\inf\limits_{\|{\mathbf v}\|_{l^{2\si+2}}^{2\si+2}=\al} J[{\mathbf v}]= \al^{\f{1}{1+\si}} \inf\limits_{\|{\mathbf v}\|_{l^{2\si+2}}^{2\si+2}=1} J[{\mathbf v}]=\al^{\f{1}{1+\si}} j(\om).
 $$
 Thus, $J[{\mathbf v^k}]\geq (\al-\epsilon)^{\f{1}{1+\si}} j(\om), J[{\mathbf w^k}]\geq (1-\al-\epsilon)^{\f{1}{1+\si}} j(\om)$. It follows that 
 $$
 J[\ubb^k]\geq  \left((\al-\epsilon)^{\f{1}{1+\si}} +(1-\al-\epsilon)^{\f{1}{1+\si}} \right)  j(\om)-C\epsilon. 
 $$
 Taking limits $\epsilon\to 0$, and then $k\to \infty$, we conclude 
 $$
 j(\om)\geq \left(\al^{\f{1}{1+\si}} +(1-\al)^{\f{1}{1+\si}} \right)  j(\om)>j(\om),
 $$
 which is also a contradiction. Thus, it remains that the sequence is tight, which implies that there is a convergent, in $l^{2\si+2}$ subsequence. We take it to be $\ubb^k$. Therefore there is $\ubb: \|\ubb\|_{l^{2\si+2}}=1$, $\lim_k \|\ubb^k-\ubb\|_{l^{2\si+2}}=0$. Since $\sup_k \|\ubb^k\|_{l^2}<\infty$, it follows that $\ubb^k\rightharpoonup \ubb$. It follows that $\{u_{n+e_j}^k-u_n^k\}_{n\in \cz^d}\rightharpoonup  \{u_{n+e_j}-u_n\}_{n\in\cz^d}$ for every $j\in \{1, \ldots, d\}$. All in all, by the lower semi-continuity of the $l^2$ norms with respect to weak convergence, 
 $$
 j(\om)=\liminf_k J[\ubb^k]\geq J[\ubb].
 $$
 On the other hand, $J[\ubb]\geq j(\om)$, whence $J[\ubb]=\lim_k J[\ubb^k]=j(\om)$ and $\ubb$ is a minimizer for \eqref{40}. Note that we have in fact proved more - since $\ubb^k\rightharpoonup \ubb$ and $\|\ubb^k\|_{l^2}\to \|\ubb\|_{l^2}$, it follows from the parallelogram identity  that in fact the stronger convergence $\lim_k \|\ubb^k-\ubb\|_{l^2}=0$ holds. In essence, this means that for all minimizing sequences for \eqref{40}, we can always find an $l^2$ (and hence $l^{2\si+2}$) convergent subsequence, which converges to a minimizer.  
 
 {Finally, the derivation of the Euler-Lagrange equation \eqref{EL:20} and the property  $\dpr{\cl_+ {\mathbf f}}{{\mathbf f}}\geq 0$ for all ${\mathbf f}: \sum_{n\in\cz^d} f_n u_n^{2\si+1}=0$, are proved in exactly the same fashion, as in the corresponding constrained problem for the normalized waves.  Specifically, consider
  $$
  g(\epsilon):= J\left[\f{{\mathbf u}^0+\epsilon{\mathbf f}}{\|{\mathbf u}^0+\epsilon{\mathbf f}\|_{l^{2\si+2}}} \right]
  $$
 The critical point condition $g'(0)=0$, which is necessary for the minimizer, implies the Euler-Lagrange equation \eqref{EL:20}.  Taking the increments ${\mathbf f}: \sum_{n\in\cz^d} f_n u_n^{2\si+1}=0$, we express the necessary condition for the minimizer, namely 
 $$
 g''(0)=\f{\dpr{\cl_+ {\mathbf f}}{{\mathbf f}}}{2}\geq 0
 $$
 which implies that $\dpr{\cl_+ {\mathbf f}}{{\mathbf f}}\geq 0$ for all ${\mathbf f}: \sum_{n\in\cz^d} f_n u_n^{2\si+1}=0$. Thus, $n(\cl_+)\leq 1$, as the operator is non-negative on a co-dimension 
 one subspace, namely $\{{\mathbf f}: {\mathbf f}\perp \{u_n^{2\si+1}\}_n\}^\perp$.  } 
 On the other hand, we still have $\dpr{\cl_+ \ubb}{\ubb}=-2\si \sum_n u_n^{2\si+1}<0$, whence $n(\cl_+)\geq 1$. Altogether, $n(\cl_+)=1$. 
 
 \subsection{Proof of Theorem \ref{theo:15}} 
  We start with a technical lemma regarding the function $j(\om)$. 
  \begin{lemma}
  	\label{le:jw} 
  	The function $j(\om)$ satisfies the bounds $\om<j(\om)<\om+2$. Moreover, it is uniformly Lipschitz in the intervals $(a, \infty), a>0$.  That is,  for each $a>0$, there is $C_a$, so that for all $a<\om_1<\om_2$, there is $|j(\om_1)-j(\om_2)|\leq C_a |\om_1-\om_2|$. 
  \end{lemma}
  {\bf Remark:} From the proof, one can take $C_a=1+\f{2}{a}$, which clearly blows up as $a\to 0+$. 
 \begin{proof}
 	Testing with the element $e_0$ yields the bound $j(\om)<J[e_0]=\om+2$. On the other hand, since 
 	$\|\ubb\|_{l^2}\geq \|\ubb\|_{l^{2\si+2}}=1$, we have that $J[\ubb]>\om$ for each $\ubb: \|\ubb\|_{l^{2\si+2}}=1$, whence $j(\om)>\om$. 
 	
 	The bound $j(\om)<\om+2$ implies that 
 	$$
 	j(\om)=\inf\limits_{\|\ubb\|_{l^{2\si+2}}=1} J[u]= \inf\limits_{\|\ubb\|_{l^{2\si+2}}=1, \|\ubb\|_{l^2}^2<1+\f{2}{\om}} J[u].
 	$$
 	Indeed, testing with $\ubb: \|\ubb\|_{l^2} >1+\f{2}{\om}$ leads to $J[u]>\om+2$, 
 	which will not be optimal. Having this in mind, let $0<a<\om_1<\om_2$. In order to test for both values $\om_1, \om_2$, it suffices to take $\ubb: \|\ubb\|_{l^{2\si+2}}=1, \|\ubb\|_{l^2}^2<1+\f{2}{\om_1}$. Since, 
 	$$
 	J_{\om_1}[\ubb]-J_{\om_2}[\ubb]=(\om_1-\om_2) \|\ubb\|_{l^2}^2,
 	$$
 	it follows that $|j(\om_1)-j(\om_2)|\leq \left(1+\f{2}{\om_1}\right)|\om_1-\om_2| \leq \left(1+\f{2}{a}\right)|\om_1-\om_2|$. 
 \end{proof}
 Having Lemma \ref{le:jw} at hand, we can finish off the proof of Theorem \ref{theo:15}. Fix $\om>0$. 
 Recall that the limit waves $\lim_j \|\ubb_{\om+\de_j}-\ubb_\om\|_{l^2}=0$.  By the continuity of $j(\om)$, we have that $\vp_n=j(\om)^{\f{1}{2\si}} u_n$ also satisfies    $\lim_j \|\vp_{\om+\de_j}-\vp_\om\|_{l^2}=0$.  Write 
 $$
 \vp_{\om+\de_j}=\vp_\om+\de_j z^j, 
 $$
 so that $\lim_j |\de_j| \|z^j\|_{l^2}=0$. Next, subtract  the profile equations 
 \begin{eqnarray*}
& &  -\De_{disc}(\vp_{\om}+\de_j z^j)+(\om+\de_j) (\vp_{\om}+\de_j z^j) - (\vp_{\om}+\de_j z^j)^{2\si+1}=0, \\
 & &  -\De_{disc}\vp_{\om}+\om \vp_{\om}  - \vp_{\om}^{2\si+1}=0. 
 \end{eqnarray*}
 and divide by $\de_j$. We obtain 
 \begin{equation}
 \label{310} 
 -\De_{disc}z^j_n +\om z^j_n - (2\si+1) \vp_n^{2\si} z^j_n=-\vp_\om - \de_j z_j +E_n^j, 
 \end{equation}
  where 
  $$
  E_n^j=\de_j^{-1}\left((\vp_{\om}+\de_j z^j)^{2\si+1} - \vp_\om^{2\si+1} - (2\si+1) \vp_n^{2\si} \de_j z^j_n\right)
  $$
  By Taylor expansions (separately in the cases $\vp_n<|\de_j| |z^j_n|$ and otherwise), it is clear that there exists a constant $C=C_\si$, so that 
  $$
  |E_n^j|\leq C_\si \vp_n^\si (|\de_j| |z^j_n|)^{\si} |z^j_n|.
  $$
  Thus, we can interpret the equation \eqref{310} in the form $\cl_+[{\mathbf z^j}]=-\vp_\om - 
  \de_j {\mathbf z_j} +{\mathbf E^j}$ This is a non-linear relation for ${\mathbf z^j}$, which we use for {\it a posteriori} estimates.  Indeed, resolving \eqref{310} leads to 
  \begin{equation}
  \label{330} 
   {\mathbf z^j}=-\cl_+^{-1}[\vp_\om]-\cl_+^{-1}(- \de_j {\mathbf z_j} +{\mathbf E^j}),
  \end{equation}
  which leads to the estimate
  $$
  \|{\mathbf z^j}\|_{l^2}\leq C \|\vp_\om\|_{l^2}+ C |\de_j| \|{\mathbf z_j}\|_{l^2}+ C \|{\mathbf E^j}\|_{l^2}\leq 
  C \|\vp_\om\|_{l^2} + C |\de_j| \|{\mathbf z_j}\|_{l^2}+ C \|{\mathbf z^j}\|_{l^2} (|\de_j| \|{\mathbf z^j}\|_{l^\infty})^\si
  $$
  Since $|\de_j| \|{\mathbf z^j}\|_{l^\infty}\leq |\de_j| \|{\mathbf z^j}\|_{l^2}\to 0$, we can hide the term $\|{\mathbf z^j}\|_{l^2} (|\de_j| \|{\mathbf z^j}\|_{l^\infty})^\si$ behind the left-hand side. This implies the estimate 
  $$
  \limsup_j \|{\mathbf z^j}\|_{l^2}\leq C \|\vp_\om\|_{l^2}. 
  $$
  Going back to the relation \eqref{330}, this yields that ${\mathbf z^j}$ converges strongly in $l^2$ to $-\cl_+^{-1}[\vp_\om]$. That is
  $$
  \lim_j \|{\mathbf z^j}+\cl_+^{-1}[\vp_\om]\|_{l^2}=0.
  $$
  Now, consider 
  \begin{eqnarray*}
  \lim_j \f{\|\vp_{\om+\de_j}\|_{l^2}^2- \|\vp_{\om}\|_{l^2}^2}{\de_j} &=&  \lim_j \f{\|\vp_{\om}+\de_j {\mathbf z^j}\|_{l^2}^2- \|\vp_{\om}\|_{l^2}^2}{\de_j}=\lim_j \left(2\dpr{\mathbf z^j}{\vp_\om}+\de_j \|{\mathbf z^j}\|_{l^2}^2\right)=\\
  &=& -2 \dpr{\cl_+^{-1} \vp_\om}{\vp_\om}. 
  \end{eqnarray*}
  By the Vakhitov-Kolokolov criteria, since $n(\cl)=n(\cl_-)+n(\cl_+)=1$, stability of $e^{i \om t} \vp_\om$ is equivalent to $\dpr{\cl_+^{-1} \vp_\om}{\vp_\om}<0$. Hence, we have proved that the wave is stable if and only if the following limit (which exists) 
  $$
  \lim_j \f{\|\vp_{\om+\de_j}\|_{l^2}^2- \|\vp_{\om}\|_{l^2}^2}{\de_j}>0.
  $$
  \subsection{Proof of Proposition \ref{prop:908}}
  The fact that $j(\om)$ is concave follows from the definition. Thus, it is twice differentiable a.e. (and we conjecture that it is smooth everywhere). Taking the profile equation for $\ubb^\om$, and differentiating in $\om$, we conclude 
  $$
  \cl_+(\p_\om \ubb_\om)=-\ubb_\om +j'(\om) \vp^{2\si+1}= -\ubb_\om-\f{j'(\om)}{\si j(\om)} \cl_+\ubb_\om.
  $$
  This implies 
  $$
  \cl_+(\p_\om \ubb_\om+\f{j'(\om)}{2\si j(\om)} \ubb_\om)=-\ubb_\om
  $$
  Taking dot product of the last equation with $\ubb^{2\si+1}$ yields 
  $$
  \|\ubb_\om\|_{l^2}^2=j'(\om).
  $$
  It follows that $\|\vp_\om\|_{l^2}^2= j'(\om) j^{\f{1}{\si}}(\om)$.

  \section{Conclusions and Future Work}

  In the present work we have revisited the original formulation
  of~\cite{Wein} regarding the identification of suitable minimizers
  in the form of normalized waves that are associated with standing
  waves of the discrete nonlinear Schr{\"o}dinger equation. We have
  provided an alternative formulation of the relevant variational
  problem,
  and we have illustrated that the zeros of a suitably defined
  function
  can offer definitive information about the stability of the relevant
  states. We have showcases the equivalence of this distinct
  formulation's
  stability conclusions in connection with the more standard VK (or
  GSS)
  stability criteria. However, a key advantage of the formulation is
  that
  it works {\it for all} values of the system parameters. In addition
  to establishing the relevant existence theorems for such minimizers,
  and identifying the associated stability conditions, we have also
  showcased the results via select numerical computations for
  different
  system parameters (including the frequency $\omega$, or the
  nonlinearity exponent $\sigma$, as well as for different spatial
  dimensionalities of the problem).

  Even within the space of even spatial solutions, however, there are
  the classes of so-called onsite and offsite solutions~\cite{jenk1}
  and the energetic difference betwen them represents the celebrated
  Peierls-Nabarro barrier~\cite{dnlsbook}. Understanding further
  the differences between such states at the level of the variational
  formulation would be something of interest for future
  considerations. On the other hand, it is well-known within this DNLS
  model
  that various excited state solutions exist, including the so-called
  twisted localized modes~\cite{dkl,kor} belong to a different class
  of
  solutions that are anti-symmetric. Yet, these waveforms are robust
  and
  can be observed to be very long-lived~\cite{dnlsbook},  hence potentially developing
  a variational formulation within such a class of functions (or over
  more complex waveforms involving a nontrivial phase distribution)
  could be  a challenging, yet interesting direction.

  \appendix

  \section{Proof of Theorem \ref{prop:PF}} 
  First, we observe that   $\cl=-\De_{disc}+{\mathbb V}$ is a bounded self-adjoint operator on $l^2(\cz^d)$. 
  It is clear that   $\cl$ is relatively compact with respect to the constant coefficient operator $-\De_{disc}$,  whence by Weyl's theorem (see for example Theorem 14.6, p. 142, \cite{HS}), their essential spectra coincide 
   $$
   \si_{ess}(\cl)=\si_{ess}(-\Delta_{disc})=\si(-\Delta_{disc})=[0, 4d].
   $$ 
  Since by assumption $\la_0(\cl)=\inf\dpr{\cl \ubb}{\ubb}<0$, it follows that $\la_0(\cl)$ is an eigenvalue. It remains to show that it is a simple eigenvalue and the corresponding eigenfunction consists of positive entries only. 
  
     According to Theorem XIII.44, p. 204, \cite{RS}, it is enough to show that the the semigroup $e^{t \cl}$,  is positivity improving. That is, for every $\uff\in l^2(\cz^d): \uff_n\geq 0, \uff\neq 0$, we need to show   that $(e^{t \cl}  \uff)_n>0, n\in \cz^d$. Clearly, $\ubb=e^{t \cl} \uff$ solves the following discrete parabolic initial value  problem 
     \begin{equation}
     	\label{par:10}
     	\p_t u_n(t)-\De_{disc} u_n(t)= V_n u_n(t), \ \ \ubb(0)=\uff, \ \    n\in \cz^d.
     \end{equation}
  By the Duhamel's formula 
  \begin{equation}
  	\label{par:12}
  \ubb(t)=e^{t \De_{disc}} \uff + \int_0^t e^{(t-s) \De_{disc}} [{\mathbb V} \ubb(s)] ds.
  \end{equation}
  Recall,  that since $e^{t \De_{disc}}$ is given by a convolution with a positive kernel\footnote{see formula \eqref{par:7} for the $\De_{disc}$ for $d=1$, the general case follows from $e^{t \De_{d}}=e^{t \De_{1}}.\ldots e^{t \De_{1}}$}  so we have the point-wise  estimate 
  \begin{equation}
  	\label{par:15} 
  	|e^{t \De_{disc}} {\mathbf h}(n)|\leq (e^{t \De_{disc}} |{\mathbf h}|)(n), n\in \cz^d,
  \end{equation}
  where we have used the notation $|{\mathbf h}|(n)=|h(n)|$. 
  
  We can now develop a Born series type expansion for $\ubb$ in \eqref{par:12}. Namely, we have 
  \begin{eqnarray*}
  	\ubb(t) &=& e^{t \De_{disc}} \uff + \int_0^t e^{(t-s) \De_{disc}} [{\mathbb V} \ubb(s)] ds = \\
  	&=& e^{t \De_{disc}} \uff + \int_0^t e^{(t-s) \De_{disc}} [{\mathbb V} e^{s \De_{disc}} \uff  ] ds+ 
  	\int_0^t e^{(t-s_1) \De_{disc}} [{\mathbb V}\int_0^{s_1} e^{(s_1-s_2) \De_{disc}} [{\mathbb V} \ubb(s_2)] ds_2 ] ds_1. 
  \end{eqnarray*}
  Iterating this further, we obtain (a formal for now)   expansion in the form 
  \begin{equation}
  	\label{par:22} 
  \ubb(t)=\sum_{j=0}^\infty I_j(\uff),
\end{equation} 
  where $I_0(\uff)=e^{t \De_{disc}} \uff, I_1(\uff)=\int_0^t e^{(t-s_1) \De_{disc}} [{\mathbb V} e^{s_1 \De_{disc}} \uff  ] ds_1$, and for $j\geq 2$, 
  $$
  I_j(\uff)=\int_0^t e^{(t-s_1) \De_{disc}} [{\mathbb V}\int_0^{s_1}   \ldots e^{(s_{j-1}-s_{j})\De_{disc}} [{\mathbb V} e^{s_j\De_{disc}} \uff] d s_j. 
  $$
  For each fixed $n\in\cz^d$, and by taking absolute values in the previous formula, and applying \eqref{par:15}  judiciously,  we obtain 
 \begin{eqnarray*}
 	 	|I_j(\uff)(n)| &\leq &  \|{\mathbb V}\|_{l^\infty(\cz^d)}^j \int_0^t \int_0^{s_1} \ldots \int_0^{s_{j-1}} e^{(t-s_1)\De_{disc}} e^{(s_1-s_2)\De_{disc}} \ldots 
 	e^{(s_{j-1}-s_{j})\De_{disc}}  e^{s_j\De_{disc}} |\uff| ds_1 \ldots ds_j\\
 	&= & \f{(\|{\mathbb V}\|_{l^\infty(\cz^d)} t)^j}{j!}  e^{t\De_{disc}} |\uff| (n).
 \end{eqnarray*}
 Summing up in $j$ yields the bound, for each $t>0, n\in \cz^d$  
 \begin{equation}
 	\label{par:30} 
 	|\ubb_n(t)|\leq  e^{\|{\mathbb V}\|_{l^\infty(\cz^d)} t} e^{t\De_{disc}} |\uff| (n).
 \end{equation}
 This also confirms that the expansion \eqref{par:22} is valid, at least in $l^2$ sense. 
 
 Having this {\it a priori} bound on $\ubb(t)$ allows us to finish the proof of the positivity improving property of the semigroup. Indeed,  by virtue of \eqref{par:15} 
 	$$
 	|\int_0^t e^{(t-s) \De_{disc}} [{\mathbb V} \ubb(s)] ds|\leq 	
 	|\int_0^t e^{(t-s) \De_{disc}} [{\mathbb V} e^{\|{\mathbb V}\|_{l^\infty(\cz^d)} s} e^{s\De_{disc}} |\uff| ] ds|\leq 
 t 	\|{\mathbb V}\|_{l^\infty(\cz^d)} e^{\|{\mathbb V}\|_{l^\infty(\cz^d)} t} e^{t\De_{disc}} |\uff|.
 	$$
 Suppose now that $\uff: \uff\geq 0, \uff\neq 0$. From \eqref{par:12}, we have 
 $$
 \ubb(t)=e^{t\De_{disc}} \uff+\int_0^t e^{(t-s) \De_{disc}} [{\mathbb V} \ubb(s)] ds\geq (1-t \|{\mathbb V}\|_{l^\infty(\cz^d)}	e^{\|{\mathbb V}\|_{l^\infty(\cz^d)} t})  e^{t\De_{disc}} \uff
 $$
 Recall that $e^{t \De} \uff>0$, as it is given by a convolution with a strictly positive sequence $K_j(t)$, see \eqref{par:7}.  Clearly, for all $t: t \|{\mathbb V}\|_{l^\infty(\cz^d)}	e^{t \|{\mathbb V}\|_{l^\infty(\cz^d)} }<1$, we have that $\ubb(t)>0$. That is, for all small enough $t$, in fact for $t<\f{1}{2  \|{\mathbb V}\|_{l^\infty(\cz^d)}}$, we have that $\ubb(t)>0$. Note that the required smallness of $t$ is only in terms of the potential ${\mathbb V}$, and importantly, it is independent on the initial data $\uff$. 
 
 Since one can write, for each t, 
 $$
 e^{t \cl} \uff= e^{\f{t}{N} \cl} \ldots e^{\f{t}{N} \cl} \uff
 $$
 with some large $N: \f{t}{N}<\f{1}{2  \|{\mathbb V}\|_{l^\infty(\cz^d)}}$, we conclude that $e^{t \cl} \uff>0$, which completes the proof of Theorem \ref{prop:PF}.


\end{document}